\documentclass{amsart}

\usepackage{graphicx} 
\usepackage{scrextend}
\usepackage{amsfonts, amsmath, amssymb,amsthm,comment}  
\usepackage{times,enumerate}
\usepackage{mathdots}%
\usepackage{nccmath}
\usepackage{mathtools}
\usepackage{ulem}
\usepackage{enumitem}
\usepackage{multicol}
    {\end{pmatrix}\end{medsize}}%
\usepackage{dsfont,bbm}
\usepackage{rotating}
\usepackage{lscape}
\usepackage{url}
\usepackage{multicol}
\usepackage{thmtools, thm-restate}
\usepackage{blkarray}
\usepackage{multicol}
\usepackage[margin=2.5cm]{geometry}
 \usepackage{hyperref}
\usepackage{cancel}
\usepackage{multirow}
\usepackage{pifont}
\usepackage{tikz}
\usetikzlibrary{topaths}
\tikzstyle{every picture}+=[remember picture,inner xsep=0,inner ysep=0.25ex]
\usetikzlibrary{calc}
\usepackage{kbordermatrix}
\renewcommand{\kbldelim}{(}
\renewcommand{\kbrdelim}{)}
\renewcommand{\kbrowstyle}{\displaystyle}
\renewcommand{\kbcolstyle}{\displaystyle}
\def\VR{\kern-\arraycolsep\strut\vrule &\kern-\arraycolsep}
\def\vr{\kern-\arraycolsep & \kern-\arraycolsep}
\makeatletter
\newcommand*{\sublabel}[1]{%
    \let\old@currentlabel\@currentlabel%
    \renewcommand{\@currentlabel}{\theenumii}%
    \label{#1}%
    \let\@currentlabel\old@currentlabel%
}
\makeatother

\allowdisplaybreaks

\DeclareMathOperator{\pr}{pr}
\DeclareMathOperator{\supp}{supp}

\newcommand{\mc}{\mathcal}

\hypersetup{
    colorlinks,
    linkcolor={blue},
    citecolor={blue},
    urlcolor={blue}
}

\DeclareMathOperator{\Rank}{rank}

\makeatletter
\def\widebreve{\mathpalette\wide@breve}
\def\wide@breve#1#2{\sbox\z@{$#1#2$}%
     \mathop{\vbox{\m@th\ialign{##\crcr
\kern0.08em\brevefill#1{0.8\wd\z@}\crcr\noalign{\nointerlineskip}%
                    $\hss#1#2\hss$\crcr}}}\limits}
\def\brevefill#1#2{$\m@th\sbox\tw@{$#1($}%
  \hss\resizebox{#2}{\wd\tw@}{\rotatebox[origin=c]{90}{\upshape(}}\hss$}
\makeatletter

\newcommand{\RR}{\mathbb R}
\newcommand{\R}{\mathbb R}

\newcommand{\uH}{\underline{H}}

\newcommand{\uh}{\underline{h}}

\newcommand{\cZ}{\mathcal Z}

\newcommand{\cC}{\mathcal C}

\newcommand{\benu}{\begin{enumerate}}
\newcommand{\eenu}{\end{enumerate}}
\newcommand{\bop}{\begin{opomba}}
\newcommand{\eop}{\end{opomba}}

\newtheorem{theorem}{Theorem}[section]

\newtheorem{proposition}[theorem]{Proposition}

\newtheorem*{problem*}{Problem}

\theoremstyle{definition}

\newtheorem{example}[theorem]{Example}

\newtheorem{remark}[theorem]{Remark}

\numberwithin{equation}{section}

\begin{document}

\title[TMP on cubic curves in Weierstrass form]{A constructive approach to the truncated moment problem on cubic curves in Weierstrass form}

\author[A.\ Bhardwaj]{Abhishek Bhardwaj}
\address{Abhishek Bhardwaj,
The Australian National University}
\email{Abhishek.Bhardwaj@anu.edu.au, Abhishek.bhardwaj.math@gmail.com}
\author[A.\ Zalar]{Alja\v z Zalar${}^{2}$}
\address{Alja\v z Zalar,
University of Ljubljana, 
Faculty of Computer and Information Science  \& 
Faculty of Mathematics and Physics, \&
Institute of Mathematics, Physics and Mechanics, Ljubljana, Slovenia.}
\email{aljaz.zalar@fri.uni-lj.si}
\thanks{${}^2$Supported by the ARIS (Slovenian Research and Innovation Agency)
research core funding No.\ P1-0288 and grants J1-50002, J1-60011, J1-70017.}


\date{\today}


\maketitle

\begin{abstract}
    In this paper, we develop a constructive solution for the pure truncated moment problem on cubic curves in Weierstrass form, establishing the existence of a representing measure whose number of atoms equals the rank of the associated moment matrix. By a recent result of Baldi, Blekherman, and Sinn, for projectively smooth curves whose projective closure has exactly one real point at infinity, the existence of such a rank-attaining atomic measure is equivalent to the existence of a representing measure; consequently, the TMP is constructively solved for this class of curves. We also present a numerical degree--$6$ example in which every minimal representing measure supported on the cubic curve requires $\Rank M(3)+1$ atoms, where $M(3)$ denotes the moment matrix. Finally, we provide a constructive solution for the symmetric case, i.e., when all moments of odd degree in $y$ vanish.
\end{abstract}

\section{Introduction}

Given a real bivariate truncated moment sequence
\begin{equation*}
\beta \equiv \beta^{(d)}=\{\beta_{ij}: i,j\ge 0,\; i+j\le d\},
\end{equation*}
of degree $d$ and a closed subset $K\subseteq\mathbb{R}^2$, the \textbf{truncated moment problem} supported on $K$
(\textbf{$K$--TMP}) asks for conditions that guarantee the existence of a positive Borel measure $\mu$ on $\mathbb{R}^2$
with $\supp(\mu)\subseteq K$ such that
\begin{equation}\label{moment-measure-cond}
\beta_{ij}=\int_{K} x^i y^j\, d\mu,
\qquad i,j\in\mathbb{Z}_+,\;\; 0\le i+j\le d.
\end{equation}
If such a measure exists, we say that $\beta$ \textbf{admits a representing measure supported on $K$}, and we call $\mu$
a \textbf{$K$--representing measure} for $\beta^{(d)}$.

Solutions to the $K$--TMP can be grouped according to how explicitly they characterize the existence of a
$K$--representing measure:
\begin{enumerate}
\item \textbf{Abstract solutions} provide existence criteria that are typically partially algorithmic. These results are
theoretically fundamental, but they can be difficult to apply in concrete numerical instances.
\item\label{concrete-sol} \textbf{Concrete solutions} give necessary and sufficient conditions that are directly verifiable in
explicit examples. Such results are important for applications, but they are often challenging to obtain.
\item\label{constructive-sol} \textbf{Constructive solutions} are concrete solutions that additionally produce a representing
measure (or an explicit procedure to construct one). They provide the strongest form of solvability and are especially
valuable in applications.
\end{enumerate}

We next survey several landmark results and representative developments for the $K$--TMP. The classical univariate cases
$K=\mathbb{R}$, $K=[a,\infty)$, and $K=[a,b]$ (with $a<b$) admit \textit{constructive} solutions in terms of the associated
Hankel matrices; see \cite{CF91}, building on certificates of positivity for polynomials on $K$ (with special cases treated
in \cite{AK62, KN77, Ioh82, ST}). In the multivariate setting, the celebrated \textit{Flat Extension Theorem} of Curto and
Fialkow \cite{CF96} characterizes solvability of the TMP via the existence of a rank-preserving positive extension of the
moment matrix. This theorem has become a primary tool for obtaining \textit{constructive} solutions for the $K$--TMP on planar
conics \cite{CF02,CF04,CF05,Fia15}, on the cubic curve $y=x^3$ \cite{Fia11}, and on the full plane $K=\mathbb{R}^2$
\cite{CY16} (the latter being first solved \textit{concretely} in \cite{fn} using convex-geometric methods).

A second major development is the \textit{Truncated Riesz--Haviland Theorem} \cite{CF08}, which reduces solvability of the
$K$--TMP to the existence of a $K$--positive extension of the associated Riesz functional. Here, $K$--positivity means
nonnegativity on every polynomial that is nonnegative on $K$. Applying this criterion in concrete settings typically
requires explicit algebraic certificates for polynomials that are nonnegative on $K$. Such certificates for \textit{planar
cubics} were established recently by Kummer and the second author \cite{KZ}, drawing on tools from algebraic geometry and
yielding \textit{concrete} TMP criteria in terms of positive semidefiniteness of two or three matrices, depending on the
geometry of the curve.

A third milestone is the \textit{Core Variety Theorem} of Blekherman and Fialkow \cite{BF20}, which characterizes the
existence of a representing measure in terms of the non-emptiness of an algebraic variety determined by the kernel of the
Riesz functional. The first concrete application of this theorem \cite{fz} provides a
simplified solution to the TMP on $y=x^3$, complementing the earlier approach in \cite{Fia11} using the Flat Extension Theorem.

Beyond these general frameworks, planar curves whose irreducible components are rational may also be treated via the
\textit{univariate TMP with gaps}, where only a subset of moments is specified. This perspective, together with
\textit{matrix completion} techniques, can substantially simplify the analysis and can yield alternative \textit{constructive}
solutions for planar conics as well as for irreducible rational and reducible planar cubics
\cite{BZ21,Zal21,Zal22b,yz,YZ25+,KZ}. This approach also applies to a quartic case, namely curves of the form
$xy=x^4+q(x)$ with $q\in\mathbb{R}[x]_{\le 3}$ \cite{yz+}. Finally, for rational planar curves defined by
$p(x,y)=y-q(x)$ or $p(x,y)=yq(x)-1$ with $q\in\mathbb{R}[x]$, solvability of the TMP is characterized by the existence of a
positive extension of the moment matrix; see \cite{Fia11,Zal23}. We also note that very recently the authors of \cite{EKT25}
solved the TMP supported on the vertices of the hypercube, motivated by applications in quantum information theory. For
further background and additional aspects of the moment problem we refer the reader to the recent monograph of
Schm\"{u}dgen \cite{sch}.
\medskip

Let \(K\) be the planar cubic $\cZ(p):=\{(x,y)\in \RR^2\colon p(x,y)=0\}$ where 
$p(x,y):=y^2-x^3+ax+b$, $a,b\in \RR$.
We call the $\cZ(p)$-TMP \textbf{$p$--pure} if the column dependence relations in the moment matrix \(M_n(\beta)\) are exactly those obtained using \textit{recursiveness} and linearity from the single column relation \(Y^2 = X^3+aX+b\textit{1}\) (see Section \ref{sec:prel} for terminology). 

In this paper we address the following problem:

\begin{problem*}
    For $p(x,y)=y^2-x^3+ax+b$, $a,b\in \RR$,
    characterize the existence of a flat extension of a moment matrix for a $\cZ(p)$--pure TMP.
\end{problem*}

As described above, a concrete solution on the curve $\cZ(p)$ 
is known \cite{KZ}. Specifically, in addition to the positivity of the moment matrix, one must verify the positivity of an additional matrix that is completely determined by the given moment data.
This result can therefore be used as a preliminary test before attempting to construct a representing measure.
However, from the application point of view, one would like to construct a representing measure.
If the test from \cite{KZ} is passed, one must then carry out a more detailed analysis based on the Flat Extension Theorem.
Moreover, by a recent result of Baldi, Blekherman, and Sinn \cite{BBS+}, if $\mathcal Z(p)$ is connected, the
homogenization of $p$ defines a smooth projective curve, and the projective closure of $\mathcal Z(p)$ has exactly one
real point at infinity, then a solution to the Problem yields a complete constructive solution to the TMP as well.
 This follows from the bound on the number of atoms in a minimal representing measure, which is equal to the rank of the moment matrix in this case. However, in case of two connected components there are moment sequences where a minimal representing measure requires \textit{one more atom} \cite{BBS+} (for a concrete numerical example see Example \ref{230809-1150} below). 
 If the first flat extension step fails in our analysis, then this is the scenario that arises.
\medskip

In this paper, after a preliminary Section \ref{sec:prel}, we solve the Problem in Section \ref{section-with-results} in terms of roots of certain quadratic polynomial with coefficients that depend entirely on the moment data.
Namely, a flat extension exists if and only if this polynomial 
has a real root (see Theorem \ref{FE-analysis}). As a corollary of the results from \cite{BBS+} this provides a constructive solution to the TMP for a large portion of projectively smooth cubic curves (see Remark \ref{rem:smooth}). In Section \ref{section-with-examples} we provide numerical examples for Theorem~\ref{FE-analysis}. That is, we construct a sequence without flat extension, an example with a unique flat extension and a sequence with no measure, demonstrating different cases that arise 
(see Examples \ref{230809-1150}-\ref{230809-2031}). Finally, in Section \ref{symmetric}, we solve the \textit{symmetric} TMP for $\cZ(p)$ constructively, i.e., $\beta_{ij}=0$
whenever $j$ is odd, using the univariate reduction technique together with the
Truncated Riesz--Haviland Theorem \cite{CF08} and the description of nonnegative polynomials on
$[0,c]\cup[d,\infty)$ from \cite{KMS05}.

\section{Preliminaries}
\label{sec:prel}

In this section we fix notation and terminology, and recall several tools needed for our main results.

\medskip

We write $\mathbb{R}^{n\times m}$ for the space of real $n\times m$ matrices. For a matrix $M$, the linear span of its
columns is its \textbf{column space}, denoted by $\mathcal{C}(M)$. The set of real symmetric $n\times n$ matrices is
denoted by $S_n$. For $A\in S_n$, the notation $A\succ 0$ (resp.\ $A\succeq 0$) means that $A$ is positive definite
(pd) (resp.\ positive semidefinite (psd)).

Throughout this section, let $d\in\mathbb{N}$ and let
\[
\beta=\beta^{(d)}=\{\beta_{ij}\}_{i,j\in\mathbb{Z}_+,\; 0\le i+j\le d}
\]
be a bivariate truncated moment sequence of degree $d$.

\subsection{Moment matrices and column relations}
\label{subsec:momMat}

Let $k:=\big\lceil \tfrac{d}{2}\big\rceil$. Consider the degree--lexicographic ordering of monomials
\[
1,\;X,\;Y,\;X^2,\;XY,\;Y^2,\;\ldots,\;X^k,\;X^{k-1}Y,\;\ldots,\;Y^k .
\]
The \textbf{moment matrix} associated with $\beta^{(d)}$ is the block matrix
\begin{equation}\label{281021-1448}
    M(k)\equiv M(k)(\beta):=
    \begin{pmatrix}
        M[0,0](\beta) & M[0,1](\beta) & \cdots & M[0,k](\beta)\\
        M[1,0](\beta) & M[1,1](\beta) & \cdots & M[1,k](\beta)\\
        \vdots & \vdots & \ddots & \vdots\\
        M[k,0](\beta) & M[k,1](\beta) & \cdots & M[k,k](\beta)
    \end{pmatrix},
\end{equation}
where, for $0\le i,j\le k$,
\[
M[i,j](\beta):=
\begin{pmatrix}
\beta_{i+j,0} & \beta_{i+j-1,1} & \beta_{i+j-2,2} & \cdots & \beta_{i,j}\\
\beta_{i+j-1,1} & \beta_{i+j-2,2} & \beta_{i+j-3,3} & \cdots & \beta_{i-1,j+1}\\
\beta_{i+j-2,2} & \beta_{i+j-3,3} & \beta_{i+j-4,4} & \cdots & \beta_{i-2,j+2}\\
\vdots & \vdots & \vdots & \ddots & \vdots\\
\beta_{j,i} & \beta_{j-1,i+1} & \beta_{j-2,i+2} & \cdots & \beta_{0,i+j}
\end{pmatrix}.
\]
If $d$ is odd, then $2k=d+1$, so some entries in the bottom-right block $M[k,k](\beta)$ would involve moments of degree
$2k>d$. Accordingly, $M(k)$ is only partially specified.

\medskip

Let
\[
\mathbb{R}[x,y]_{\le k}:=\{p\in\mathbb{R}[x,y]: \deg p\le k\},
\]
where $\deg p$ denotes total degree. For a polynomial
$p(x,y)=\sum_{i,j} a_{ij}x^iy^j\in\mathbb{R}[x,y]_{\le k}$,
we define its \textbf{evaluation} $p(X,Y)$ in $M(k)$ by
\[
p(X,Y):=\sum_{i,j} a_{ij}\,X^iY^j,
\]
where $X^iY^j$ denotes the column of $M(k)$ indexed by the monomial $X^iY^j$ in the chosen ordering. Thus
$p(X,Y)\in \mathcal{C}(M(k))$. If $p(X,Y)=0$, we say that $p$ is a \textbf{column relation} of $M(k)$.

Recall (see, e.g., \cite{CF96}) that if $\beta$ admits a representing measure $\mu$ with
$$
\supp(\mu)\subseteq \mathcal{Z}(p):=\{(x,y)\in\mathbb{R}^2:\ p(x,y)=0\},
$$
then $p$ is a column relation of $M(k)$.

We say that $M(k)$ is \textbf{recursively generated} if, whenever $p,q,pq\in\mathbb{R}[x,y]_{\le k}$ and $p$ is a column
relation of $M(k)$, then $pq$ is also a column relation of $M(k)$. Finally, $M(k)$ is called \textbf{$p$--pure} if its
column relations are exactly those obtained from $p$ by linearity and recursiveness.

\subsection{The Riesz functional}

The linear functional $L_{\beta}:\mathbb{R}[x,y]_{\le d}\to\mathbb{R}$ defined by
\[
    L_{\beta}(p)
    :=
    \sum_{\substack{i,j\in\mathbb{Z}_+\\ 0\le i+j\le d}} a_{ij}\,\beta_{ij},
    \qquad
    \text{where }
    p(x,y)=\sum_{\substack{i,j\in\mathbb{Z}_+\\ 0\le i+j\le d}} a_{ij}x^iy^j,
\]
is called the \textbf{Riesz functional} associated with the sequence $\beta$.

\subsection{Affine linear transformations}\label{affine linear-trans}

The existence of representing measures is invariant under invertible affine linear transformations of the form
\begin{equation}\label{alt}
    \phi(x,y)=(\phi_1(x,y),\phi_2(x,y))
    :=
    (a+bx+cy,\ d+ex+fy), \qquad (x,y)\in\mathbb{R}^2,
\end{equation}
where $a,b,c,d,e,f\in\mathbb{R}$ and $bf-ce\neq 0$.

Let $L_{\beta}$ be the Riesz functional associated with $\beta$, and define the transformed sequence
\[
\widetilde{\beta}\equiv \phi(\beta):=\{\widetilde{\beta}_{ij}\}_{i,j\in\mathbb{Z}_+,\; 0\le i+j\le 2k}
\]
by
\[
\widetilde{\beta}_{ij}:=L_{\beta}\!\bigl(\phi_1(x,y)^i\,\phi_2(x,y)^j\bigr).
\]

\begin{proposition}[{\cite[Proposition 1.9]{CF05}}] \label{251021-2254}
	Assume the notation above.
	\begin{enumerate}
		\item $M(k)(\beta)$ is positive semidefinite if and only if $M(k)(\widetilde \beta)$ is positive semidefinite.
		\item $\Rank M(k)(\beta)=\Rank M(k)(\widetilde\beta)$.
		\item $M(k)(\beta)$ is recursively generated if and only if $M(k)(\widetilde\beta)$ is recursively generated.
		\item\label{291021-2333} $\beta$ admits an $r$--atomic $K$--representing measure if and only if $\widetilde \beta$ admits an $r$--atomic $\phi(K)$--representing measure.
	\end{enumerate}
\end{proposition}

\section{
Existence of a $(\Rank M(n))$--atomic representing measure}
\label{section-with-results}

In this section we characterize when a sequence $\beta\equiv \beta^{(2n)}$, whose moment matrix $M(n)$ is positive semidefinite and $(y^2-x^3-ax-b)$--pure (see Subsection \ref{subsec:momMat}) with $n\ge 3$, admits a $\big(\operatorname{rank} M(n)\big)$--atomic representing measure.\\

Since $M(n)$ is $(y^2-x^3-ax-b)$--pure, its column dependence relations are precisely linear combinations of the relations
\[
    X^iY^{2+j}=X^{3+i}Y^j+aX^{1+i}Y^j+bX^iY^j,
    \qquad i,j\ge 0,\;\; i+j\le n-3.
\]
In particular, every column indexed by a monomial involving $X^3$ can be reduced to a linear combination of columns indexed
by monomials with $X$--exponent at most $2$. It follows that
\[
\operatorname{rank} M(n)=3n,
\]
and one may choose a basis of the column space $\mathcal{C}\!\left(M(n)\right)$ of the form
\[
\mathcal{B}
=
\{1,\;X,\;Y\}\ \cup\ 
    \bigcup_{k=2}^{n}\{X^{2}Y^{k-2},\; XY^{k-1},\;Y^{k}\}.
\]

Clearly, if $\beta^{(2n)}$ admits a representing measure, then it admits a positive semidefinite, recursively generated
moment-matrix extension of the form
\begin{equation}\label{moment-mat-ext-n+1}
    M(n+1)\equiv
    \begin{pmatrix}
        M(n) & B(n+1)\\
        B(n+1)^{t} & C(n+1)
    \end{pmatrix}.
\end{equation}
Observe that, by recursiveness, all blocks of $B(n+1)$ are determined by the data $\beta^{(2n)}$ except for the block
$B[n,n+1]$. The entries of this block involve moments of total degree $2n+1$, and among these, only three are not fixed by
the recursion. We denote the corresponding undetermined moments by $\theta,\phi,\psi$. Consequently, $B[n,n+1]$ has the
form
\begin{equation}\label{B-block}
\kbordermatrix{
 & X^{n+1} & X^{n}Y & \dotsc & \dotsc & X^{2}Y^{n-1} & XY^{n} & Y^{n+1} \\
\textit{1}       & \beta_{n+1,0}   & \beta_{n,1}     & \dotsc & \dotsc & \beta_{2,n-1}   & \beta_{1,n}     & \beta_{0,n+1} \\
X               & \beta_{n+2,0}   & \beta_{n+1,1}   & \cdots & \cdots & \beta_{3,n-1}   & \beta_{2,n}     & \beta_{1,n+1} \\
\vdots          & \vdots          & \vdots          &        &        & \vdots          & \vdots          & \vdots \\
\vdots          & \vdots          & \vdots          &        &        & \vdots          & \vdots          & \beta_{3,2n-2} \\
X^{2}Y^{n-2}    & \beta_{n+3,n-2} & \beta_{n+2,n-1} & \cdots & \cdots & \beta_{4,2n-3}  & \beta_{3,2n-2}  & \theta \\
XY^{n-1}        & \beta_{n+2,n-1} & \beta_{n+1,n}   & \cdots & \cdots & \beta_{3,2n-2}  & \theta          & \phi \\
Y^{n}           & \beta_{n+1,n}   & \beta_{n,n+1}   & \cdots & \beta_{3,2n-2} & \theta & \phi & \psi
}.
\end{equation}

Indeed, the degree $(2n+1)$ moments of the form $\beta_{i+3,\,2n-i-2}$ for $i=0,\ldots,2n-2$ are determined recursively by
the relation $Y^2=X^3+aX+b$, namely,
\[
    \beta_{i+3,\,2n-i-2}
    =
    \beta_{i,\,2n-i}
    - a\,\beta_{i+1,\,2n-i-2}
    - b\,\beta_{i,\,2n-i-2}.
\]

Next, consider the compressions (to the rows indexed by $\mathcal{B}$) of the columns of $M(n+1)$ indexed by
$X^{2}Y^{n-1}$, $XY^{n}$, and $Y^{n+1}$. These compressed columns have the respective forms
\[
    [X^{2}Y^{n-1}]_{\mathcal{B}}
    =
    \begin{pmatrix} w\\ \theta \end{pmatrix},\qquad
    [XY^{n}]_{\mathcal{B}}
    =
    \begin{pmatrix} q\\ \theta\\ \phi \end{pmatrix},\qquad
    [Y^{n+1}]_{\mathcal{B}}
    =
    \begin{pmatrix} p\\ \theta\\ \phi\\ \psi \end{pmatrix},
\]
where $w:=(w_1,\ldots,w_{3n-1})^{t}$, $q:=(q_1,\ldots,q_{3n-2})^{t}$, and
$p:=(p_1,\ldots,p_{3n-3})^{t}$. Moreover, each entry $w_i,q_i,p_i$ is a moment $\beta_{ij}$ with
$i,j\ge 0$ and $i+j\le 2n$.

Define the functions $\phi_{\ast}(\theta)$, $\psi_{\ast}(\theta,\phi)$, and $Q(\theta,\phi,\psi)$ by
\begin{align}\label{phi-ast}
\begin{split}
    \phi_{\ast}(\theta)
        & :=
            [X^{2}Y^{n-1}]_{\mathcal{B}}^{t}\,
            [M(n)]_{\mathcal{B}}^{-1}\,
            [X^{2}Y^{n-1}]_{\mathcal{B}}
            +a\,\beta_{2,2n-2}+b\,\beta_{1,2n-2},\\
    \psi_{\ast}(\theta,\phi)
        & :=
            [XY^{n}]_{\mathcal{B}}^{t}\,
            [M(n)]_{\mathcal{B}}^{-1}\,
            [X^{2}Y^{n-1}]_{\mathcal{B}}
            +a\,\beta_{1,2n-1}+b\,\beta_{0,2n-1},\\
    Q(\theta,\phi,\psi)
        & :=
            [Y^{n+1}]_{\mathcal{B}}^{t}\,
            [M(n)]_{\mathcal{B}}^{-1}\,
            [X^{2}Y^{n-1}]_{\mathcal{B}}
            -
            [XY^{n}]_{\mathcal{B}}^{t}\,
            [M(n)]_{\mathcal{B}}^{-1}\,
            [XY^{n}]_{\mathcal{B}}.
\end{split}
\end{align}
Finally, set
\begin{equation}\label{def-R}
    R(\theta):=Q\!\bigl(\theta,\phi_{\ast}(\theta),\psi_{\ast}(\theta,\phi_{\ast}(\theta))\bigr).
\end{equation}

The next proposition shows that $R(\theta)$ has a particularly simple form.

\begin{proposition}\label{quadratic-R-theta}
    The function $R(\theta)$ is a quadratic polynomial in $\theta$.
\end{proposition}

The following theorem resolves the existence of a $(\Rank M(n))$--atomic
$\mc Z(p)$--representing measure for $\beta$.

\begin{theorem}\label{FE-analysis}
Assume the notation introduced above. Let $M(n)\succeq 0$ be $p$--pure, where
\[
p(x,y)=y^{2}-x^{3}-ax-b,\qquad a,b\in\mathbb{R}.
\]
Then the following statements are equivalent:
\begin{enumerate}
\item\label{ft-theorem-pt1}
    $\beta$ admits a $\bigl(\operatorname{rank} M(n)\bigr)$--atomic $\mathcal{Z}(p)$--representing measure.
\item\label{ft-theorem-pt2}
    $M(n)$ admits a flat extension $M(n+1)$.
\item\label{ft-theorem-pt3}
    The quadratic polynomial $R(\theta)$ has a real root.
\end{enumerate}
Moreover, suppose that $R(\theta)=R_2\theta^2+R_1\theta+R_0\in\mathbb{R}[\theta]$ is nonzero and has a real root, and set
$\Delta:=R_1^2-4R_0R_2$. Then:
\begin{enumerate}
\item If $\Delta=0$, there is a unique $\bigl(\operatorname{rank} M(n)\bigr)$--atomic representing measure.
\item If $\Delta>0$, there are exactly two $\bigl(\operatorname{rank} M(n)\bigr)$--atomic representing measures.
\end{enumerate}
\end{theorem}

\begin{remark}
\label{rem:smooth}
    By \cite[Theorem~7.2.2]{BBS+}, if $\mathcal Z(p)$ is connected, if the projective curve
    \[
    \{(x,y,z)\in\mathbb C^3:\ zy^2-x^3-axz^2-bz^3=0\}
    \]
    is smooth, and if the projective closure of $\mathcal Z(p)$ has exactly one real point at infinity, then
    conditions {(1)--(3)} in Theorem~\ref{FE-analysis} are equivalent to the existence of a
    $\mathcal Z(p)$--representing measure for $\beta$. In particular, this yields a complete constructive
    solution to the TMP in this setting. If, on the other hand, $\mathcal Z(p)$ has two connected
    components, then a representing measure may require $(\Rank M(n)+1)$ atoms; see \cite[Theorem~7.2.4]{BBS+}.
    We illustrate this situation with a numerical example in the next section; see Example~\ref{230809-1150}.
    We leave a complete constructive solution to the TMP on such curves $\mathcal Z(p)$ as an open problem.
    There are two natural approaches: (a) characterize the existence of a second \textit{flat} moment-matrix extension, or
    (b) extend simultaneously the two matrices identified in \cite{KZ} to positive semidefinite ones, one of which is \textit{flat}.
    Both directions appear to require a substantial and technically demanding analysis.
\end{remark}

\medskip
First we prove Proposition~\ref{quadratic-R-theta}.

\begin{proof}[Proof of Proposition~\ref{quadratic-R-theta}]
Write the compression of $M(n)$ to the basis $\mc B$ in block form as
\[
    [M(n)]_{\mc B}
    =
    \begin{pmatrix}
        M & x\\
        x^{t} & \beta_{0,2n}
    \end{pmatrix}.
\]
By \cite[p.~3144]{Fia11},
\[
    [M(n)]_{\mc B}^{-1}
    =
    \begin{pmatrix}
        P & v\\
        v^{t} & \varepsilon
    \end{pmatrix},
\]
where
\[
    \varepsilon := \bigl(\beta_{0,2n}-x^{t}M^{-1}x\bigr)^{-1},
    \qquad
    v := -\varepsilon M^{-1}x,
    \qquad
    P := M^{-1}\bigl(I+\varepsilon xx^{t}M^{-1}\bigr).
\]
Let $r_i$ denote the $i$th row of $P$. Define
\[
    [M(n)]_{\mc B}^{-1}[X^{2}Y^{n-1}]_{\mc B}
    =:
    \bigl(c_{1}(\theta),\ldots,c_{3n}(\theta)\bigr)^{t}
    \equiv c(\theta).
\]
Since $[X^{2}Y^{n-1}]_{\mc B}=(w^{t},\theta)^{t}$, where $w\in\RR^{3n-1}$ is fixed, we obtain
\[
    c_i(\theta)=
    \begin{cases}
        \langle r_i^{t},w\rangle + v_i\,\theta, & 1\le i\le 3n-1,\\[0.2em]
        \langle v,w\rangle + \varepsilon\,\theta, & i=3n.
    \end{cases}
\]
Using this representation, we compute
\begin{align}\label{LHS-1}
\begin{split}
\phi_{\ast}(\theta)
&=
\left\langle
    \begin{pmatrix} w\\ \theta \end{pmatrix},
    c(\theta)
\right\rangle
+ a\beta_{2,2n-2}+b\beta_{1,2n-2}\\
&=
\sum_{i=1}^{3n-1} w_i c_i(\theta)
+\theta\bigl(\langle v,w\rangle+\varepsilon\theta\bigr)
+ a\beta_{2,2n-2}+b\beta_{1,2n-2}\\
&=
\varepsilon\theta^{2}
+2\langle v,w\rangle\,\theta
+\sum_{i=1}^{3n-1}\langle r_i^{t},w\rangle w_i
+a\beta_{2,2n-2}+b\beta_{1,2n-2}\\
&=: f_2\theta^{2}+f_1\theta+f_0\in\R[\theta].
\end{split}
\end{align}

Next, partition $P$ as
\[
    P=
    \begin{pmatrix}
        Q & u\\
        u^{t} & \beta_{2,2n-2}
    \end{pmatrix},
    \qquad
    u=(u_i)_{i=1}^{3n-2},
\]
and denote by $s_i$ the $i$th row of $Q$. Define
\[
    [M(n)]_{\mc B}^{-1}[XY^{n}]_{\mc B}
    =:
    \bigl(d_{1}(\theta,\phi),\ldots,d_{3n}(\theta,\phi)\bigr)^{t}
    \equiv d(\theta,\phi).
\]
Since $[XY^{n}]_{\mc B}=(q^{t},\theta,\phi)^{t}$ with $q\in\R^{3n-2}$ fixed, we have
\[
d_i(\theta,\phi)=
\begin{cases}
    \langle s_i^{t},q\rangle + u_i\,\theta + v_i\,\phi, & 1\le i\le 3n-2,\\[0.2em]
    \langle u,q\rangle + \beta_{2,2n-2}\theta + v_{3n-1}\phi, & i=3n-1,\\[0.2em]
    \Big\langle v,\begin{pmatrix} q\\ \theta\end{pmatrix}\Big\rangle+\varepsilon\phi, & i=3n.
\end{cases}
\]

We now compute $\psi_{\ast}(\theta,\phi)$ using \eqref{phi-ast}. Using 
$[X^{2}Y^{n-1}]_{\mc B}=(w^{t},\theta)^{t}$ and
$c(\theta)=[M(n)]_{\mc B}^{-1}[X^{2}Y^{n-1}]_{\mc B}$,
\begin{align}\label{LHS-2}
\begin{split}
\psi_{\ast}(\theta,\phi)
&=
\Big\langle
    \begin{pmatrix} q\\ \theta\\ \phi \end{pmatrix},
    c(\theta)
\Big\rangle
+a\beta_{1,2n-1}+b\beta_{0,2n-1}\\
&=
\sum_{i=1}^{3n-2}q_ic_i(\theta)
+\theta c_{3n-1}(\theta)
+\phi c_{3n}(\theta)
+a\beta_{1,2n-1}+b\beta_{0,2n-1}\\
&=
\varepsilon\,\phi\theta + \langle v,w\rangle\,\phi
+v_{3n-1}\theta^{2}
+\Bigl(\langle r_{3n-1}^{t},w\rangle+\langle (v_1,\ldots,v_{3n-2})^{t},q\rangle\Bigr)\theta \\
&\hspace{2cm}
+\sum_{i=1}^{3n-2}q_i\langle r_i^{t},w\rangle
+a\beta_{1,2n-1}+b\beta_{0,2n-1}\\
&=: j_{11}\phi\theta+j_{10}\phi+j_{02}\theta^2+j_{01}\theta+j_{00}
\in\R[\theta,\phi].
\end{split}
\end{align}
Substituting $\phi=\phi_{\ast}(\theta)=f_2\theta^2+f_1\theta+f_0$ yields
\begin{align*}
\psi_{\ast}(\theta,\phi_{\ast}(\theta))
&=
j_{11}\theta\bigl(f_2\theta^2+f_1\theta+f_0\bigr)
+j_{10}\bigl(f_2\theta^2+f_1\theta+f_0\bigr)
+j_{02}\theta^2+j_{01}\theta+j_{00}\\
&=
j_{11}f_2\theta^3+
(j_{11}f_1+j_{10}f_2+j_{02})\theta^2
+(j_{11}f_0+j_{10}f_1+j_{01})\theta
+(j_{10}f_0+j_{00})\\
&=: j_3\theta^3+j_2\theta^2+j_1\theta+j_0.
\end{align*}

Using
$[Y^{n+1}]_{\mc B}=(p^{t},\theta,\phi,\psi)^{t}$ with $p\in\R^{3n-3}$ fixed, we obtain
\begin{align}\label{RHS-3}
\begin{split}
[Y^{n+1}]_{\mc B}^{t}[M(n)]_{\mc B}^{-1}[X^2Y^{n-1}]_{\mc B}
&=
\left\langle
    (p_1,\ldots,p_{3n-3},\theta,\phi,\psi)^{t},
    c(\theta)
\right\rangle \\
&=
\sum_{i=1}^{3n-3}p_ic_i(\theta)
+\theta c_{3n-2}(\theta)
+\phi c_{3n-1}(\theta)
+\psi c_{3n}(\theta)\\
&=
\varepsilon\,\psi\theta+\langle v,w\rangle\psi
+v_{3n-1}\phi\theta+\langle r_{3n-1}^{t},w\rangle\phi
+v_{3n-2}\theta^2 \\
&
+\Bigl(\langle r_{3n-2}^{t},w\rangle+\langle (v_1,\ldots,v_{3n-3})^{t},p\rangle\Bigr)\theta
+\sum_{i=1}^{3n-3}\langle r_i^{t},w\rangle p_i\\
&=: k(\theta,\phi,\psi)\in\R[\theta,\phi,\psi],
\end{split}
\end{align}
where
\[
k(\theta,\phi,\psi)
=
k_{101}\psi\theta+k_{100}\psi
+k_{011}\phi\theta+k_{010}\phi
+k_{002}\theta^2+k_{001}\theta+k_{000}.
\]
Substituting $\phi=\phi_{\ast}(\theta)$ and $\psi=\psi_{\ast}(\theta,\phi_{\ast}(\theta))$ gives
\begin{align*}
k(\theta,\phi_{\ast}(\theta),\psi_{\ast}(\theta,\phi_{\ast}(\theta)))
&=
k_4\theta^4+k_3\theta^3+k_2\theta^2+k_1\theta+k_0.
\end{align*}

Using
$[XY^{n}]_{\mc B}=(q^{t},\theta,\phi)^{t}$ and $d(\theta,\phi)=[M(n)]_{\mc B}^{-1}[XY^{n}]_{\mc B}$,
\begin{align}\label{LHS-3}
\begin{split}
[XY^{n}]_{\mc B}^{t}[M(n)]_{\mc B}^{-1}[XY^{n}]_{\mc B}
&=
\left\langle
    (q_1,\ldots,q_{3n-2},\theta,\phi)^{t},
    d(\theta,\phi)
\right\rangle \\
&=
\sum_{i=1}^{3n-2}q_id_i(\theta,\phi)
+\theta d_{3n-1}(\theta,\phi)
+\phi d_{3n}(\theta,\phi)\\
&=
\varepsilon\phi^{2}
+2v_{3n-1}\phi\theta
+2\langle (v_i)_{i=1}^{3n-2},q\rangle\,\phi
+\beta_{2,2n-2}\theta^2
\\
&\hspace{2cm}+2\langle u,q\rangle\theta
+\sum_{i=1}^{3n-2}\langle s_i^{t},q\rangle q_i \\
&=: \ell(\theta,\phi)\in\R[\theta,\phi],
\end{split}
\end{align}
where
\[
\ell(\theta,\phi)
=
\ell_{20}\phi^2+\ell_{11}\phi\theta+\ell_{10}\phi
+\ell_{02}\theta^2+\ell_{01}\theta+\ell_{00}.
\]
Substituting $\phi=\phi_{\ast}(\theta)$ yields
\begin{align*}
\ell(\theta,\phi_{\ast}(\theta))
&=
(\ell_{20}f_2^2)\theta^4
+\bigl(2\ell_{20}f_2f_1+\ell_{11}f_2\bigr)\theta^3
+
\left(\ell_{20}(2f_2f_0+f_1^2)+\ell_{11}f_1+\ell_{10}f_2+\ell_{02}
\right)\theta^2\\
&\hspace{1cm}
+\bigl(
2\ell_{20}f_{1}f_0
+
\ell_{11}f_0
+
\ell_{10}f_1+\ell_{01}
\bigr)\theta
+\bigl(\ell_{20}f_0^2+\ell_{10}f_0+\ell_{00}\bigr)\\
&=: \ell_4\theta^4+\ell_3\theta^3+\ell_2\theta^2+\ell_1\theta+\ell_0.
\end{align*}

Therefore,
\[
R(\theta)
=
k(\theta,\phi_{\ast}(\theta),\psi_{\ast}(\theta,\phi_{\ast}(\theta)))
-
\ell(\theta,\phi_{\ast}(\theta))
=
\sum_{i=0}^{4}(k_i-\ell_i)\theta^i.
\]
To complete the proof it suffices to show that $k_4-\ell_4=0$ and $k_3-\ell_3=0$.
A direct computation gives
\begin{align*}
k_4-\ell_4
&=
k_{101}j_3-\ell_{20}f_2^2
=
\varepsilon(j_{11}f_2)-\varepsilon\varepsilon^2
=\varepsilon^3-\varepsilon^3
=0,\\[0.4em]
k_3-\ell_3
&=
(k_{101}j_2+k_{100}j_3+k_{011}f_2)
-
(2\ell_{20}f_2f_1+\ell_{11}f_2)\\
&=
\Bigl(\varepsilon(j_{11}f_1+j_{10}f_2+j_{02})
+\langle v,w\rangle\,j_{11}f_2
+v_{3n-1}\varepsilon\Bigr)
-
\bigl(4\varepsilon^2\langle v,w\rangle+2v_{3n-1}\varepsilon\bigr)\\
&=
\bigl(4\varepsilon^2\langle v,w\rangle+2v_{3n-1}\varepsilon\bigr)
-
\bigl(4\varepsilon^2\langle v,w\rangle+2v_{3n-1}\varepsilon\bigr)
=0.
\end{align*}
Hence $R(\theta)$ is a polynomial of degree at most $2$, i.e., a quadratic polynomial in $\theta$.
\end{proof}

Now we are ready to prove Theorem~\ref{FE-analysis}.

\begin{proof}[Proof of Theorem~\ref{FE-analysis}]
The equivalence of \eqref{ft-theorem-pt1} and \eqref{ft-theorem-pt2} is exactly the Flat Extension Theorem
\cite[Theorem~7.10]{CF96}.

We now prove the equivalence of \eqref{ft-theorem-pt2} and \eqref{ft-theorem-pt3}. We begin by showing that there is a
three-parameter family of blocks $B(n+1)$ of the form \eqref{B-block} for which the matrix
\[
    \widetilde{M(n)}
    :=
    \begin{pmatrix}
        M(n) & B(n+1)
    \end{pmatrix}
\]
has moment structure, is recursively generated, and satisfies
\begin{equation}\label{230802-2355}
    \mathcal{C}\bigl(B(n+1)\bigr)\subseteq \mathcal{C}\bigl(M(n)\bigr).
\end{equation}
By recursive generation, $\widetilde{M(n)}$ must satisfy the column relations
\begin{equation}\label{230803-0220}
    X^{i+3}Y^{j}
    =
    X^{i}Y^{j+2}-aX^{i+1}Y^{j}-bX^{i}Y^{j},
    \qquad i,j\ge 0,\;\; i+j+3\le n.
\end{equation}
At the level of moments, these relations are equivalent to
\begin{equation}\label{230802-2332}
    \beta_{i+3,j}
    =
    \beta_{i,j+2}-a\beta_{i+1,j}-b\beta_{i,j},
    \qquad i,j\ge 0,\;\; i+j+3\le 2n+1.
\end{equation}
Since $M(n)$ is recursively generated, \eqref{230802-2332} already holds whenever $i+j+3\le 2n$. We therefore define the
new moments of degree $2n+1$ by
\[
    \beta_{i+3,\,2n-i-2}
    :=
    \beta_{i,\,2n-i}
    -a\beta_{i+1,\,2n-i-2}
    -b\beta_{i,\,2n-i-2},
    \qquad i=0,\ldots,2n-2.
\]
The remaining three moments of degree $2n+1$,
\[
    \beta_{2,2n-1},\qquad \beta_{1,2n},\qquad \beta_{0,2n+1},
\]
are free parameters; we denote them by $\theta,\phi,\psi$.

\medskip
\noindent\textbf{Claim 1.}
Fix $\theta,\phi,\psi$. Then \eqref{230802-2355} holds.

\smallskip
\noindent\textit{Proof of Claim 1.}
Since the moment relations \eqref{230802-2332} are equivalent to the column relations \eqref{230803-0220},
it follows that
\[
    X^{n+1},\,X^{n}Y,\,\ldots,\,X^{3}Y^{n-2}\in \mc C\big(M(n)\big).
\]
It therefore remains to show that the columns
\[
    X^{2}Y^{n-1},\qquad XY^{n},\qquad Y^{n+1}
\]
(of $B(n+1)$) belong to $\mc C\big(M(n)\big)$.

The columns indexed by $\mc B$ form a basis of $\mc C(M(n))$. Since $M(n)\succeq 0$ and
$\Rank M(n)=|\mc B|$, the compression $[M(n)]_{\mc B}$ is invertible.
Let $i_0\in\{0,1,2\}$ and set $C:=X^{i_0}Y^{n+1-i_0}$.
To prove $C\in \mc C(M(n))$ it suffices to verify that
\begin{equation}\label{230803-0010}
    C
    =
    [M(n)]_{\ast,\mc B}\,[M(n)]_{\mc B}^{-1}\,[C]_{\mc B},
\end{equation}
where $[M(n)]_{\ast,\mc B}$ denotes the restriction of $M(n)$ to the columns indexed by $\mc B$.
By construction, \eqref{230803-0010} holds for every row indexed by a monomial in $\mc B$.

Now let $X^{i}Y^{j}$ be a row indexed by a monomial not in $\mc B$, i.e., $i\ge 3$, $j\ge 0$,
and $i+j\le n$. We prove by induction on $i$ that
\begin{equation}\label{remains-to-prove}
    [C]_{X^{i}Y^{j}}
    =
    [M(n)]_{X^{i}Y^{j},\mc B}\,[M(n)]_{\mc B}^{-1}\,[C]_{\mc B}.
\end{equation}

\smallskip
\noindent\textit{Base step: $i=3$.}
Then $j\le n-3$, and by recursive generation,
\(
X^{3}Y^{j}=Y^{j+2}-aXY^{j}-bY^{j}
\)
holds in both $M(n)$ and in the column $C$. Hence
\begin{align*}
[C]_{X^{3}Y^{j}}
&=
[C]_{Y^{j+2}}-a[C]_{XY^{j}}-b[C]_{Y^{j}}\\
&=
\Bigl([M(n)]_{Y^{j+2},\mc B}-a[M(n)]_{XY^{j},\mc B}-b[M(n)]_{Y^{j},\mc B}\Bigr)
[M(n)]_{\mc B}^{-1}[C]_{\mc B}\\
&=
[M(n)]_{X^{3}Y^{j},\mc B}\,[M(n)]_{\mc B}^{-1}\,[C]_{\mc B},
\end{align*}
where the second equality uses \eqref{remains-to-prove} for the rows $Y^{j+2},XY^{j},Y^{j}\in\mc B$.

\smallskip
\noindent\textit{Induction step.}
Assume \eqref{remains-to-prove} holds for all rows $X^{i}Y^{j}$ with $i\le i_{1}$ and
$0\le j\le n-i$, where $i_1\ge 3$. Let $i=i_{1}+1$ and $0\le j\le n-i_{1}-1$.
By recursive generation,
\[
    X^{i_{1}+1}Y^{j}
    =
    X^{i_{1}-2}Y^{j+2}-aX^{i_{1}-1}Y^{j}-bX^{i_{1}-2}Y^{j}
\]
holds in $M(n)$ and in the column $C$, so
\begin{align*}
[C]_{X^{i_{1}+1}Y^{j}}
&=
[C]_{X^{i_{1}-2}Y^{j+2}}
-a[C]_{X^{i_{1}-1}Y^{j}}
-b[C]_{X^{i_{1}-2}Y^{j}}\\
&=
\Bigl([M(n)]_{X^{i_{1}-2}Y^{j+2},\mc B}
-a[M(n)]_{X^{i_{1}-1}Y^{j},\mc B}
-b[M(n)]_{X^{i_{1}-2}Y^{j},\mc B}\Bigr)
[M(n)]_{\mc B}^{-1}[C]_{\mc B}\\
&=
[M(n)]_{X^{i_{1}+1}Y^{j},\mc B}\,[M(n)]_{\mc B}^{-1}\,[C]_{\mc B},
\end{align*}
where the second equality follows from the induction hypothesis applied to the rows
$X^{i_{1}-2}Y^{j+2}$, $X^{i_{1}-1}Y^{j}$, and $X^{i_{1}-2}Y^{j}$.
This completes the induction and proves \eqref{remains-to-prove}, hence \eqref{230803-0010}.
Therefore $C\in\mc C(M(n))$, and Claim~1 follows.
\hfill$\blacksquare$
\medskip

In every moment matrix extension \eqref{moment-mat-ext-n+1} of $M(n)$, the lower-right $4\times 4$
submatrix of $C(n+1)$ has the form
\begin{equation}\label{230809-0922}
    \kbordermatrix{
    & X^3 Y^{n-2} & X^2Y^{n-1} & XY^n & Y^{n+1}\\
    X^3 Y^{n-2} &
        \beta_{6,2n-4} & \beta_{5,2n-3} & \beta_{4,2n-2} & \beta_{3,2n-1} \\
    X^{2}Y^{n-1}&
        \beta_{5,2n-3} & \beta_{4,2n-2} & \beta_{3,2n-1} & \beta_{2,2n}\\
    XY^n &
        \beta_{4,2n-2} & \beta_{3,2n-1} & \beta_{2,2n} & \beta_{1,2n+1}\\
    Y^{n+1} &
        \beta_{3,2n-1} & \beta_{2,2n} & \beta_{1,2n+1} & \beta_{0,2n+2}
    }.
\end{equation}

\medskip
\noindent\textbf{Claim 2.}
\[
    \beta_{4,2n-2}=\phi-a\beta_{2,2n-2}-b\beta_{1,2n-2},
    \qquad
    \beta_{3,2n-1}=\psi-a\beta_{1,2n-1}-b\beta_{0,2n-1}.
\]

\smallskip
\noindent\textit{Proof of Claim 2.}
In any recursively generated extension $M(n+1)$, the relations \eqref{230803-0220} must hold for every
choice of the parameters $\theta,\phi,\psi$.
In particular,
\[
    X^{3}Y^{n-2} = Y^{n} - aXY^{n-2} - bY^{n-2}\qquad\text{in }M(n+1),
\]
which implies
\begin{align}
\begin{split}\label{RHS-1}
 \beta_{4,2n-2}
 &=
 [XY^{n}]^t_{\mc B}\,[M(n)]_{\mc B}^{-1}\,[X^{3}Y^{n-2}]_{\mc B}\\
 &=
 [XY^{n}]^t_{\mc B}\,[M(n)]_{\mc B}^{-1}
\Bigl(
    [Y^{n}]_{\mc B}-a[XY^{n-2}]_{\mc B}-b[Y^{n-2}]_{\mc B}
\Bigr)\\[0.2em]
&=
\phi-a\beta_{2,2n-2}-b\beta_{1,2n-2},
\end{split}
\end{align}
and, similarly,
\begin{align}
\begin{split}\label{RHS-2}
\beta_{3,2n-1}
&=
[Y^{n+1}]^t_{\mc B}\,[M(n)]_{\mc B}^{-1}\,[X^{3}Y^{n-2}]_{\mc B}\\
&=
[Y^{n+1}]^t_{\mc B}\,[M(n)]_{\mc B}^{-1}
\Bigl(
    [Y^{n}]_{\mc B}-a[XY^{n-2}]_{\mc B}-b[Y^{n-2}]_{\mc B}
\Bigr)\\[0.2em]
&=
\psi-a\beta_{1,2n-1}-b\beta_{0,2n-1}.
\end{split}
\end{align}
This proves Claim~2.
\hfill$\blacksquare$

\medskip
Thus, in any extension $M(n+1)$ the entries $\beta_{4,2n-2}$ and $\beta_{3,2n-1}$ are determined once
$\phi$ and $\psi$ are fixed, respectively.
Fix $\theta,\phi,\psi$ and define
\begin{equation}\label{230809-0931}
    \widehat C
    :=
    B(n+1)^{t}\,[M(n)]_{\mc B}^{-1}\,[B(n+1)]_{\mc B}.
\end{equation}
Let
\begin{equation}\label{230809-1220}
    \widehat{M}
    =
    \begin{pmatrix}
        M(n) & B(n+1)\\
        B(n+1)^{t} & \widehat C
    \end{pmatrix}.
\end{equation}
By construction, the column relations \eqref{230803-0220} hold in $\widehat M$; hence $\widehat M$ has a
moment structure in all rows and columns $X^{i}Y^{j}$ with $i,j\ge 0$ and $i+j\le n$, and also in the
columns $X^{n+1},X^{n}Y,\ldots,X^{3}Y^{n-2}$.
It remains to analyze the moment structure in the rows and columns
$X^{2}Y^{n-1}$, $XY^{n}$, and $Y^{n+1}$.
The lower-right $4\times 4$ submatrix of $\widehat C$ has the form
\begin{equation}\label{230809-1218}
\kbordermatrix{
    & X^3Y^{n-2} & X^2 Y^{n-1} & XY^n & Y^{n+1}\\
X^{3}Y^{n-2} &
    \widehat C_{n-1,n-1} & \widehat C_{n,n-1} & \widehat C_{n+1,n-1} & \widehat C_{n+2,n-1}\\
X^2 Y^{n-1} &
    \widehat C_{n,n-1} & \widehat C_{n,n} & \widehat C_{n+1,n} & \widehat C_{n+2,n}\\
XY^n &
    \widehat C_{n+1,n-1} & \widehat C_{n+1,n} & \widehat C_{n+1,n+1} & \widehat C_{n+2,n+1}\\
Y^{n+1} &
   \widehat C_{n+2,n-1} & \widehat C_{n+2,n} & \widehat C_{n+2,n+1} & \widehat C_{n+2,n+2}
}.
\end{equation}

\medskip
\noindent\textbf{Claim 3.}
The matrix $\widehat M$ has a moment structure if and only if
\begin{align}\label{230809-0926}
\begin{split}
\widehat C_{n,n} &= \widehat C_{n+1,n-1},\\
\widehat C_{n+1,n} &= \widehat C_{n+2,n-1},\\
\widehat C_{n+1,n+1} &= \widehat C_{n+2,n},
\end{split}
\end{align}
or, equivalently, if and only if the following identities hold:
\begin{align}
    [X^{2}Y^{n-1}]^t_{\mc B}\,[M(n)]_{\mc B}^{-1}\,[X^{2}Y^{n-1}]_{\mc B}
 &=
    [XY^{n}]^t_{\mc B}\,[M(n)]_{\mc B}^{-1}\,[X^{3}Y^{n-2}]_{\mc B},
    \label{fialkow2.11}\\
    [XY^{n}]^t_{\mc B}\,[M(n)]_{\mc B}^{-1}\,[X^{2}Y^{n-1}]_{\mc B}
 &=
    [Y^{n+1}]^t_{\mc B}\,[M(n)]_{\mc B}^{-1}\,[X^{3}Y^{n-2}]_{\mc B},
\label{fialkow2.12}\\
    [XY^{n}]^t_{\mc B}\,[M(n)]_{\mc B}^{-1}\,[XY^{n}]_{\mc B}
 &=
    [Y^{n+1}]^t_{\mc B}\,[M(n)]_{\mc B}^{-1}\,[X^{2}Y^{n-1}]_{\mc B}.
\label{fialkow2.13}
\end{align}

\smallskip
\noindent\textit{Proof of Claim 3.}
The equalities in \eqref{230809-0926} are necessary for $\widehat M$ to have a moment structure, by
comparison with the universal form \eqref{230809-0922} of the corresponding $4\times4$ block
in any moment matrix extension.
The equivalence between \eqref{230809-0926} and \eqref{fialkow2.11}--\eqref{fialkow2.13}
follows directly from the definition \eqref{230809-0931} of $\widehat C$.
\hfill$\blacksquare$

\medskip
\noindent\textbf{Claim 4.}
The system \eqref{fialkow2.11}--\eqref{fialkow2.13} is equivalent to the existence of a real root of
the polynomial $R(\theta)$.

\smallskip
\noindent\textit{Proof of Claim 4.}
Using \eqref{RHS-1}, the identity \eqref{fialkow2.11} holds if and only if $\phi=\phi_{\ast}(\theta)$,
where $\phi_{\ast}(\theta)$ is defined in \eqref{phi-ast}.
Using \eqref{RHS-2}, the identity \eqref{fialkow2.12} holds if and only if $\psi=\psi_{\ast}(\theta,\phi)$,
where $\psi_{\ast}(\theta,\phi)$ is defined in \eqref{phi-ast}.
Finally, \eqref{fialkow2.13} holds if and only if $Q(\theta,\phi,\psi)=0$, where $Q$ is given by \eqref{phi-ast}.
Consequently, the system \eqref{fialkow2.11}--\eqref{fialkow2.13} is satisfied if and only if
$R(\theta)=0$ (cf.\ \eqref{def-R}) for some real $\theta$.
\hfill$\blacksquare$

\medskip
By Proposition~\ref{quadratic-R-theta}, the function $R(\theta)$ is a quadratic polynomial.
This establishes the equivalence of \eqref{ft-theorem-pt2} and \eqref{ft-theorem-pt3}.
The concluding statement follows by counting distinct real roots of the quadratic polynomial $R(\theta)$,
which is determined by the sign of its discriminant.
\end{proof}


\section{Numerical examples for the results from Section \ref{section-with-results}}
\label{section-with-examples}

The examples\footnote{The used tool for generating examples was Mathematica. The source code for the examples is available at  \url{https://github.com/ZalarA/TMP-Weierstrass}.} in this section concern sextic moment data
\[
\beta^{(6)}=\bigl(\beta_{ij}\bigr)_{i,j\in\mathbb{Z}_+,\; i+j\le 6},
\]
for which the associated moment matrix $M(3)$ (indexed by monomials of degree at most $3$ in degree--lexicographic order)
has the form
\[
\kbordermatrix{
& 1 & X & Y & X^2 & XY & Y^2 & X^3 & X^2Y & XY^2 & Y^3\\
1   & \beta_{00} & \beta_{10} & \beta_{01} & \beta_{20} & \beta_{11} & \beta_{02} & \beta_{30} & \beta_{21} & \beta_{12} & \beta_{03}\\
X   & \beta_{10} & \beta_{20} & \beta_{11} & \beta_{30} & \beta_{21} & \beta_{12} & \beta_{40} & \beta_{31} & \beta_{22} & \beta_{13}\\
Y   & \beta_{01} & \beta_{11} & \beta_{02} & \beta_{21} & \beta_{12} & \beta_{03} & \beta_{31} & \beta_{22} & \beta_{13} & \beta_{04}\\
X^2 & \beta_{20} & \beta_{30} & \beta_{21} & \beta_{40} & \beta_{31} & \beta_{22} & \beta_{50} & \beta_{41} & \beta_{32} & \beta_{23}\\
XY  & \beta_{11} & \beta_{21} & \beta_{12} & \beta_{31} & \beta_{22} & \beta_{13} & \beta_{41} & \beta_{32} & \beta_{23} & \beta_{14}\\
Y^2 & \beta_{02} & \beta_{12} & \beta_{03} & \beta_{22} & \beta_{13} & \beta_{04} & \beta_{32} & \beta_{23} & \beta_{14} & \beta_{05}\\
X^3 & \beta_{30} & \beta_{40} & \beta_{31} & \beta_{50} & \beta_{41} & \beta_{32} & \beta_{60} & \beta_{51} & \beta_{42} & \beta_{33}\\
X^2Y& \beta_{21} & \beta_{31} & \beta_{22} & \beta_{41} & \beta_{32} & \beta_{23} & \beta_{51} & \beta_{42} & \beta_{33} & \beta_{24}\\
XY^2& \beta_{12} & \beta_{22} & \beta_{13} & \beta_{32} & \beta_{23} & \beta_{14} & \beta_{42} & \beta_{33} & \beta_{24} & \beta_{15}\\
Y^3 & \beta_{03} & \beta_{13} & \beta_{04} & \beta_{23} & \beta_{14} & \beta_{05} & \beta_{33} & \beta_{24} & \beta_{15} & \beta_{06}
}.
\]

The following example shows that not every $p$--pure moment sequence admits a
$(\Rank M(n))$--atomic $\mathcal{Z}(p)$--representing measure.

\begin{example}\label{230809-1150}
Let
\[
a=-\frac{524287}{262144}
\qquad\text{and}\qquad
b=1.
\]
We consider the (unnormalized) moment matrix $M(3)$ generated by the $10$ atoms
\[
(x_i,y_i):=\left(\frac{1}{i},\sqrt{\frac{1}{i^{3}}+\frac{a}{i}+b}\right),
\qquad
(x_{5+i},y_{5+i}):=(x_i,-y_i),
\qquad i=1,\ldots,5.
\]
More explicitly,
\begin{tiny}
\begin{align*}
x_1 &= 1, \qquad
    y_1=\frac{1}{512},
&\qquad
x_3 &= \frac{1}{3}, \qquad
    y_3=\frac{\sqrt{\frac{262144}{3}}}{2560},
&\qquad
x_5 &= \frac{1}{5}, \qquad
    y_5=\frac{\sqrt{\frac{19922969}{5}}}{1024},\\[0.2em]
x_2 &= \frac{1}{2}, \qquad
    y_2=\frac{\sqrt{\frac{65537}{2}}}{512},
&\qquad
x_4 &= \frac{1}{4}, \qquad
    y_4=\frac{\sqrt{540673}}{1024}.
\end{align*}
\end{tiny}

Since all atoms lie on the curve $\mathcal{Z}(p)$, where
\[
p(x,y)=y^2-x^3-ax-b,
\]
it follows that $M(3)$ satisfies the column relation
\[
Y^{2}=X^{3}+aX+b\textit{1},
\]
and, by construction, $\beta$ admits a $\mathcal{Z}(p)$--representing measure.

The moments are:
\begin{tiny}
\begin{align*}
\beta_{00} &= 10,
&\beta_{10} &= \frac{137}{30},
&\beta_{20} &= \frac{5269}{1800},\\[0.2em]
\beta_{02} &= \frac{5729564777}{1769472000},
&\beta_{30} &= \frac{256103}{108000},
&\beta_{12} &= \frac{92678280949}{106168320000},\\[0.2em]
\beta_{40} &= \frac{14001361}{6480000},
&\beta_{22} &= \frac{1642891248263}{6370099200000},
&\beta_{04} &= \frac{9873214333730598229}{6262062317568000000},\\[0.2em]
\beta_{50} &=
        \frac{806108207}{388800000},
&\beta_{32} &=
        \frac{32309537807281}{382205952000000},
&\beta_{14} &=
        \frac{145734450995651475623}{375723739054080000000},\\[0.2em]
\beta_{60} &=
        \frac{47463376609}{23328000000},
&\beta_{42} &=
        \frac{705210728359247}{22932357120000000},
&\beta_{24} &=
        \frac{2279209309048894035601}{22543424343244800000000},\\[0.2em]
\beta_{06} &=
        \frac{18376199902549241303491562183}{22161087866383368192000000000}.
\end{align*}
\end{tiny}

A \textit{Mathematica} computation shows that
\[
R(\theta)=R_2\theta^{2}+R_1\theta+R_0
\]
satisfies $R_1=0$ and
\begin{tiny}
\begin{align*}
R_2
&=
\frac{2518293870123022495609405302939763563092225775041011300}
     {121617394571298435190879906936561845321520470769},
\qquad
&R_0
&=
\frac{1850617701610280004960481}
     {11427409289822154604482953216000}.
\end{align*}
\end{tiny}
In particular, $R(\theta)$ has no real root. Therefore, by Theorem~\ref{FE-analysis}, $\beta$ does not admit a
$(\Rank M(n))$--atomic $\mathcal{Z}(p)$--representing measure. 
\end{example}

\medskip

The motivation for the following example is to show that there is a polynomial $R(\theta)$
for the $(y^2-x^3)$--pure data with a positive semidefinite $M(3)$, which has only one real root and hence one flat extension.

\begin{example}
\label{240607-1046}
For the moments
\begin{tiny}
\begin{align*}
    \beta_{00} 
    & = 9,
    &\beta_{10} 
    &=\frac{1077749}{705600},\\[0.2em]
    \beta_{01} 
    & =\frac{78708473}{65856000} 
    &\beta_{20} 
    & = 
     \frac{538589354801}{497871360000},\\[0.2em]
     \beta_{11}
    & =
     \frac{144545256245731}{139403980800000},
     &\beta_{02} 
    & =\beta_{30}=\frac{357389058474664049}{351298031616000000},\\[0.2em]
    \beta_{21}
    &= \frac{99184670126682733619}{98363448852480000000},
    &
   \beta_{12}&
   = \beta_{40}=\frac{248886558707571775009601}{247875891108249600000000},\\
    \beta_{03}& =\beta_{31}=\frac{7727182467755471289426059}{7711694390034432000000000},
    &\beta_{22} & =\beta_{50}=\frac{175075181098169912564190119249}{174901228765980917760000000000},\\
    \beta_{13} & =\beta_{41} =\frac{48996545626484461837262019724819}{4897234405447465697280000000000
   0},
    &\beta_{04} &
    =\beta_{32}=\beta_{60}
    =\frac{123440676633749862544386180451074401}{123410307017276135571456000
   000000000},\\
    \beta_{23} &=\beta_{51}=
    \frac{34559126310550692454895975393101963331}{3455488596483731796000768
   0000000000000},
    &\beta_{14} &= \beta_{42} =
    \frac{87083646015637763331978612833793555534449}{8707831263139004125921
   9353600000000000000},\\[0.2em]
    \beta_{05}
    &= \beta_{33}=
    \frac{8127557778983392187744541237413631640158673}{81273091789297371841
   93806336000000000000000}
    &\beta_{24}
    &=
   \frac{61443396372281352421374004069586174620242179201}{6144245739270881
   3112505175900160000000000000000}
    \\[0.2em]
    \beta_{15}
    &=
    \frac{17204019459451797001978107903201762504044841332131}{1720388806995
   8467671501449252044800000000000000000}\\
    \beta_{06}
    &=\frac{43353963430456212574182304484730467470511653009909649}{4335379793
   6295338532183652115152896000000000000000000},
\end{align*}
\end{tiny}
a \textit{Mathematica} 
computation 
reveals that $M(3)$ is positive semidefinite and $(y^2-x^3)$--pure. Moreover,
\[
    R(\theta)=-\frac{(a-b\theta)^2}{c},
\]
where
\begin{tiny}
\begin{align*}
a&=12139086586077884004193854007024034872245209422802987219,\\
b&=12139063422162694789011422592242810880000000000000000000,\\
c&=90641965446620657513025511098298018324395010735907478615575494656\cdot 10^{36}.
\end{align*}
\end{tiny}
\noindent Therefore, $R(\theta)$ has a unique real root of multiplicity two, and $\beta$ has a unique
$\bigl(\Rank M(3)\bigr)$--atomic $\cZ(y^2-x^3)$--representing measure. Furthermore, by
\cite[Corollary~4.4]{Zal21}, $\beta$ admits a unique representing measure.
Indeed, the curve $y^2=x^3$ admits the parametrization $(x,y)=(t^2,t^3)$, $t\in\RR$.
Define a univariate sequence $\gamma=\{\gamma_k\}_{0\le k\le 18}$ by
$\gamma_{2i+3j}:=\widehat{\beta}_{ij}$, $(i,j\in\mathbb Z_+,\ i+j\le 6)$.
Note that $\gamma_1$ is not specified by this rule.
Observe that $\beta$ admits a representing measure supported on $y^2=x^3$ if and only if there exists
$\gamma_1\in\RR$ such that the sequence $\gamma$ admits a representing measure supported on $\RR$.
Since the Hankel matrix $\bigl(\gamma_{i+j}\bigr)_{i,j=1}^{9}$ is singular, it follows that $\gamma$
has a unique representing measure.
\end{example}


The motivation for the following example is to compute the polynomial $R(\theta)$
for the $(y^2-x^3)$--pure data with a positive semidefinite $M(3)$, which do not come from a measure.

\begin{example}
\label{230809-2031}
For the moments
\begin{tiny}
\begin{align*}
    \beta_{00} 
    & = 8,
    &\beta_{10} 
    &=\frac{372149}{705600},\\[0.2em]
    \beta_{01} 
    & =\frac{12852473}{65856000}
    &\beta_{20} 
    & = 
     \frac{40717994801}{497871360000},\\[0.2em]
     \beta_{11}
    & =
     \frac{5141275445731}{139403980800000},
     &\beta_{02} 
    & =\beta_{30}=\frac{6091026858664049}{351298031616000000},\\[0.2em]
    \beta_{21}
    &= \frac{821221274202733619}{98363448852480000000},
   &\beta_{12}&
   = \beta_{40}=\frac{1010667599322175009601}{247875891108249600000000},\\[0.2em]
    \beta_{03}& =\beta_{31}=\frac{15488077721039289426059}{7711694390034432000000000},
    &\beta_{22} & =\beta_{50}=\frac{173952332188994804190119249}{174901228765980917760000000000},
    \\[0.2em]
    \beta_{13} & =\beta_{41} =\frac{24201572009804864462019724819}{48972344054474656972800000000000},
    &\beta_{04} &
    =\beta_{32}=\beta_{60}
    =\frac{30369616473726972930180451074401}{123410307017276135571456000000000000},
    \\[0.2em]
    \beta_{23} &=\beta_{51}=
    \frac{4240345713374494888295393101963331}{
34554885964837317960007680000000000000},
    &\beta_{14} &= \beta_{42} =
    \frac{5333384247722072759259233793555534449}{
87078312631390041259219353600000000000000},
    \\[0.2em]
    \beta_{05}
    &= \beta_{33}=
    \frac{248600053655003550734901413631640158673}
{8127309178929737184193806336000000000000000}
    &\beta_{24}
    &=
    \frac{938979572539308868828169426174620242179201}{
61442457392708813112505175900160000000000000000}
    \\[0.2em]
    \beta_{15}
    &=
    \frac{131389493329330476658651156962504044841332131}{
17203888069958467671501449252044800000000000000000},\\
    \beta_{06}
    &=\frac{4335545287790407895217017581092861070511653009909649}{
43353797936295338532183652115152896000000000000000000
},
\end{align*}
\end{tiny}
a \textit{Mathematica} computation 
reveals the following. First, $M(3)$ is positive semidefinite and $(y^2-x^3)$--pure. The polynomial
$
R(\theta)=R_2\theta^{2}+R_1\theta+R_0
$
satisfies $R_2=R_1=0$ and $R_0=-16257024$. Hence $R(\theta)$ is a negative constant, and therefore
no $(\Rank M(3))$--atomic representing measure for $\beta$ exists.

Moreover, $\beta$ admits no representing measure. This follows from the solution of the truncated moment
problem supported on $y^2=x^3$; see \cite[Corollary~4.3]{Zal21}. Namely, as in
Example~\ref{240607-1046}, let $\gamma=\{\gamma_k\}_{0\le k\le 18}$ denote the associated univariate
sequence. By \cite[Corollary~4.4]{Zal21}, in the $(y^2-x^3)$--pure case, $\beta$ admits a representing
measure if and only if the Hankel matrix $\bigl(\gamma_{i+j}\bigr)_{i,j=1}^{8}$ is positive definite and
$\bigl(\gamma_{i+j}\bigr)_{i,j=1}^{9}$ is positive semidefinite. In the present situation,
$\bigl(\gamma_{i+j}\bigr)_{i,j=1}^{8}$ is singular, and consequently no representing measure exists for
$\gamma$, and hence neither for $\beta$.
\hfill$\blacksquare$
\end{example}


\section{Constructive solution to the symmetric $\cZ(p)$--TMP}
\label{symmetric}

Let $n\ge 3$, let
$
    \beta:=\{\beta_{ij}\}_{i,j\in\mathbb{Z}_+,\; i+j\le 2n}
$
be a bivariate moment sequence of degree $2n$.
Assume that the moment matrix $M(n)$ is positive semidefinite and $p$--pure, where 
    $$p(x,y)=y^2-x^3-ax-b, \quad a,b\in\mathbb{R}.$$
In addition, suppose that
\[
    \beta_{ij}=0 \quad\text{whenever } j \text{ is odd}.
\]
We call such data \textbf{symmetric} and refer to the corresponding truncated moment problem as the
\textbf{symmetric TMP}. In this section we give a constructive solution to the symmetric
$\mathcal{Z}(p)$--TMP.\\

The following proposition reduces the symmetric $\mathcal Z(p)$--TMP to a univariate TMP.
\medskip

\begin{proposition}\label{prop:one_step_reduction}
Let $n\ge 3$ and let $\beta=\{\beta_{ij}\}_{i,j\in\mathbb Z_+,\, i+j\le 2n}$ be a bivariate truncated
moment sequence of degree $2n$ satisfying the symmetry condition
$\beta_{ij}=0$ 
whenever $j$ is odd.
Let
\[
p(x,y)=y^2-x^3-ax-b,\qquad a,b\in\mathbb R,
\]
and denote by $\pr_1:\mathbb R^2\to\mathbb R$ the first coordinate projection, $\pr_1(x,y):=x$.

\smallskip
Define the following:

\smallskip
\begin{enumerate}
\item
Let $\widetilde\beta=\{\widetilde\beta_{ij}\}_{i,j\in\mathbb Z_+,\, i+2j\le 2n}$ be given by
\[
\widetilde\beta_{ij}:=\beta_{i,2j}.
\]
Set
\[
\widetilde K:=\{(x,z)\in\mathbb R\times[0,\infty):\ z=x^3+ax+b\}.
\]

\item
Let $\phi:\mathbb R^2\to\mathbb R^2$ be defined by
\[
\phi(x,z)=(x,z-ax-b),
\]
and let $\widehat\beta:=\phi(\widetilde\beta)$ (see \eqref{affine linear-trans}). Then
\[
\phi(\widetilde K)=\{(x,x^3):\ x\in \pr_1(\widetilde K)\}.
\]

\item
Let $\gamma=\{\gamma_t\}_{0\le t\le 3n}$ be the univariate truncated sequence defined by
\[
\gamma_{i+3j}:=\widehat\beta_{ij}
\qquad (i,j\in\mathbb Z_+,\ i+2j\le 2n).
\]

\item
Let $x_1$ be the smallest real root of $x^3+ax+b$, and define $\widetilde\phi:\mathbb R\to\mathbb R$
by $\widetilde\phi(x)=x-x_1$. Set $\widetilde\gamma:=\widetilde\phi(\gamma)$ and define
\[
E
:=\widetilde\phi\big(\pr_1(\phi(\widetilde K))\big)
=\widetilde\phi\big(\pr_1(\widetilde K)\big).
\]
Equivalently,
\begin{equation}
\label{def:E}
E=
\begin{cases}
[0,\infty), & \text{if $x^3+ax+b$ has exactly one real root $x_1$},\\[2mm]
[0,x_2-x_1]\cup[x_3-x_1,\infty),
& \text{if $x^3+ax+b$ has three real roots $x_1<x_2<x_3$}.
\end{cases}
\end{equation}
\end{enumerate}

\smallskip
Then the following are equivalent:
\begin{enumerate}
\item[(i)] $\beta$ admits a representing measure supported on
\[
\mathcal Z(p)=\{(x,y)\in\mathbb R^2:\ p(x,y)=0\}.
\]
\item[(ii)] $\widetilde\gamma$ admits a representing measure supported on $E$.
\end{enumerate}

\smallskip
Moreover, given a representing measure for any one of the sequences above, one can construct representing
measures for all the others via the explicit transformations described in {(1)--(4)}.
\end{proposition}

\begin{proof}
We prove (i)$\Leftrightarrow$(ii) by exhibiting a chain of equivalences
$
\text{(i)}\Leftrightarrow\text{(iii)}\Leftrightarrow\text{(iv)}\Leftrightarrow\text{(v)}\Leftrightarrow\text{(ii)},
$
where:
\begin{enumerate}
\item[(iii)] $\widetilde\beta$ admits a $\widetilde K$--representing measure;
\item[(iv)] $\widehat\beta$ admits a $\phi(\widetilde K)$--representing measure;
\item[(v)] $\gamma$ admits a $\pr_1(\phi(\widetilde K))$--representing measure.
\end{enumerate}

By Richter's theorem \cite{Ric}, it suffices to consider finitely atomic representing measures in all poinst (i)-(v).

\smallskip\noindent
\textit{(i)$\Leftrightarrow$(iii).}
Apply the substitution $z=y^2$. If $\mu$ is a representing measure for $\beta$ supported on
$\mathcal Z(p)$, then its push-forward under $(x,y)\mapsto(x,y^2)$ is supported on $\widetilde K$ and
represents $\widetilde\beta$, since $\widetilde\beta_{ij}=\beta_{i,2j}$.
Conversely, given a representing measure $\widetilde\mu$ for $\widetilde\beta$ supported on
$\widetilde K$, one lifts it to a measure on $\mathcal Z(p)$ by replacing each atom $(x_r,z_r)$ with the
pair $(x_r,\pm\sqrt{z_r})$, each carrying half of the mass. This lifting reproduces all even $y$--moments by construction, while the odd
$y$--moments vanish; the symmetry assumption $\beta_{ij}=0$ for odd $j$ is precisely what ensures that
the lifted measure matches $\beta$.

\smallskip\noindent
\textit{(iii)$\Leftrightarrow$(iv).}
This follows from the invariance of the truncated moment problem under affine changes of variables (see Proposition \ref{251021-2254}):
push-forward and pull-back via the affine map $\phi(x,z)=(x,z-ax-b)$ transport representing measures
between $\widetilde\beta$ on $\widetilde K$ and $\widehat\beta=\phi(\widetilde\beta)$ on
$\phi(\widetilde K)$.

\smallskip\noindent
\textit{(iv)$\Leftrightarrow$(v).}
On $\phi(\widetilde K)$ the second coordinate satisfies $u=x^3$. Hence each monomial $x^i u^j$
restricts to $x^{i+3j}$, and the bivariate moments $\widehat\beta_{ij}$ coincide with the univariate
moments $\gamma_{i+3j}$ on the projected set $\pr_1(\phi(\widetilde K))$.

\smallskip\noindent
\textit{(v)$\Leftrightarrow$(ii).}
Finally, apply the univariate affine change of variables $\widetilde\phi(x)=x-x_1$ from item~(4):
push-forward by $\widetilde\phi$ maps representing measures for $\gamma$ supported on
$\pr_1(\phi(\widetilde K))$ onto representing measures for $\widetilde\gamma=\widetilde\phi(\gamma)$
supported on $E=\widetilde\phi(\pr_1(\phi(\widetilde K)))$, and conversely.

\smallskip
Combining these equivalences yields (i)$\Leftrightarrow$(ii). Moreover, the constructions above are
explicit at each step, so composing them provides representing measures for all intermediate sequences.
\end{proof}

\medskip
Thus, to solve the symmetric $\mathcal{Z}(p)$--TMP explicitly, it suffices to solve the univariate TMP on
\[
[0,\infty)
\quad\text{and on}\quad
[0,c]\cup[d,\infty)\quad\text{for some }0<c<d.
\]
The case $[0,\infty)$ is treated in \cite{CF91}. The case $[0,c]\cup[d,\infty)$ can be handled using the
truncated Riesz--Haviland theorem \cite{CF08} together with the description of nonnegative polynomials on
$[0,c]\cup[d,\infty)$ from \cite{KMS05}. Before stating the results, we introduce additional notation.

\bigskip

Let $m,n\in\mathbb{N}$ with $m\le \frac{n}{2}$, and let $\gamma=\{\gamma_t\}_{0\le t\le n}$ be a real univariate
truncated sequence. Define the Hankel matrix and the associated column vector by
\[
    H_m(\gamma):=\bigl(\gamma_{i+j}\bigr)_{i,j=0}^{m},
    \qquad
    h_m(\gamma):=
    \begin{pmatrix}
        \gamma_{m}\\ \gamma_{m+1}\\ \vdots\\ \gamma_{2m-1}
    \end{pmatrix}.
\]
Let $\mathcal{I}$ and $\mathcal{E}$ denote the identity and shift operators, respectively:
\[
    \mathcal{I}\gamma:=(\gamma_0,\ldots,\gamma_n),
    \qquad
    \mathcal{E}\gamma:=(\gamma_1,\ldots,\gamma_n).
\]
If $p(x)=a_0+a_1x+\cdots+a_kx^k$ with $k\le n$, define (by a mild abuse of notation) the linear operator
\[
    p(\mathcal{E})
    :=a_0\mathcal{I}+a_1\mathcal{E}+\cdots+a_k\mathcal{E}^k
    :\mathbb{R}^{n+1}\to \mathbb{R}^{n+1-k},
\]
where each summand is truncated to its first $n+1-k$ coordinates and $\mathcal{E}^k$ denotes $k$ successive shifts.

\medskip
Fix $c,d\in\mathbb{R}$ with $0<c<d$. We introduce the following (localizing) Hankel matrices:
\begin{align*}
    \uH_{2m}(\gamma)
    &:=H_m(\gamma),\\[0.6em]
    \uH_{2m+1}^{[0,\infty)}(\gamma)
    &:=H_m(\mathcal{E}\gamma),\\[0.6em]
    \uH_{2m}^{(-\infty,c]\cup[d,\infty)}(\gamma)
    &:=H_{m-1}\!\bigl((\mathcal{E}-c\mathcal{I})(\mathcal{E}-d\mathcal{I})\gamma\bigr)
     =H_{m-1}\!\bigl((cd\mathcal{I}-(c+d)\mathcal{E}+\mathcal{E}^2)\gamma\bigr),\\[0.6em]
    \uH_{2m+1}^{[0,c]\cup[d,\infty)}(\gamma)
    &:=H_{m-1}\!\bigl(\mathcal{E}(\mathcal{E}-c\mathcal{I})(\mathcal{E}-d\mathcal{I})\gamma\bigr)
     =H_{m-1}\!\bigl((cd\mathcal{E}-(c+d)\mathcal{E}^2+\mathcal{E}^3)\gamma\bigr).
\end{align*}

\begin{remark}
The index $r$ in $\uH_r$ indicates the largest index for which $\gamma_r$ appears in the definition of the
corresponding (localizing) Hankel matrix.
\end{remark}

These matrices admit the block decompositions
\begin{align*}
    \uH_{2m}(\gamma)
    &=
    \begin{pmatrix}
        \uH_{2m-2}(\gamma) & \uh_{2m-1}(\gamma)\\[0.4em]
        \uh_{2m-1}(\gamma)^{T} & \gamma_{2m}
    \end{pmatrix},\\[0.6em]
    \uH_{2m+1}^{[0,\infty)}(\gamma)
    &=
    \begin{pmatrix}
        \uH_{2m-1}^{[0,\infty)}(\gamma) & \uh_{2m}^{[0,\infty)}(\gamma)\\[0.4em]
        \uh_{2m}^{[0,\infty)}(\gamma)^{T} & \gamma_{2m+1}
    \end{pmatrix},\\[0.6em]
    \uH_{2m}^{(-\infty,c]\cup[d,\infty)}(\gamma)
    &=
    \begin{pmatrix}
        \uH_{2m-2}^{(-\infty,c]\cup[d,\infty)}(\gamma) & \uh_{2m-1}^{(-\infty,c]\cup[d,\infty)}(\gamma)\\[0.4em]
        \uh_{2m-1}^{(-\infty,c]\cup[d,\infty)}(\gamma)^{T}
        & cd\gamma_{2m-2}-(c+d)\gamma_{2m-1}+\gamma_{2m}
    \end{pmatrix},\\[0.6em]
    \uH_{2m+1}^{[0,c]\cup[d,\infty)}(\gamma)
    &=
    \begin{pmatrix}
        \uH_{2m-1}^{[0,c]\cup[d,\infty)}(\gamma) & \uh_{2m}^{[0,c]\cup[d,\infty)}(\gamma)\\[0.4em]
        \uh_{2m}^{[0,c]\cup[d,\infty)}(\gamma)^{T}
        & cd\gamma_{2m-1}-(c+d)\gamma_{2m}+\gamma_{2m+1}
    \end{pmatrix},
\end{align*}
where
\begin{align*}
    \uh_{2m-1}(\gamma)
    &=h_m(\gamma),\\[0.4em]
    \uh_{2m}^{[0,\infty)}(\gamma)
    &=h_m(\mathcal{E}\gamma),\\[0.4em]
    \uh_{2m-1}^{(-\infty,c]\cup[d,\infty)}(\gamma)
    &=h_{m-1}\!\bigl((\mathcal{E}-c\mathcal{I})(\mathcal{E}-d\mathcal{I})\gamma\bigr),\\[0.4em]
    \uh_{2m}^{[0,c]\cup[d,\infty)}(\gamma)
    &=h_{m-1}\!\bigl(\mathcal{E}(\mathcal{E}-c\mathcal{I})(\mathcal{E}-d\mathcal{I})\gamma\bigr).
\end{align*}

Let $\gamma \equiv \gamma^{(2k)} := \{\gamma_t\}_{t=0}^{2k}$ be a univariate truncated sequence.
Assume that the Hankel matrix
$H_k(\gamma)$ is positive semidefinite. 
Following \cite{CF91} (see also \cite[Proposition~2.2]{CF91}), the \textbf{rank} of $\gamma$, denoted by $\Rank \gamma$,
is defined as
\[
\Rank \gamma :=
\begin{cases}
k+1, 
& \text{if } H_k(\gamma) \text{ is nonsingular},\\[2mm]
\min\bigl\{\, i \in \mathbb{N}\cup\{0\} \;:\; H_i(\gamma) \text{ is singular}\,\bigr\},
& \text{otherwise}.
\end{cases}
\]

A sequence $\gamma \equiv \gamma^{(2k)}$ with $r := \Rank \gamma$ is called
\textbf{positively recursively generated} if $H_{r-1}(\gamma)\succ 0$ and, upon defining
\[
(\varphi_0,\ldots,\varphi_{r-1})^{T}
:= H_{r}(\gamma)^{-1}\,(\gamma_r,\ldots,\gamma_{2r-1})^{T},
\]
the recursion
\begin{equation}\label{recursion}
\gamma_j
=\varphi_0 \gamma_{j-r} + \cdots + \varphi_{r-1}\gamma_{j-1},
\qquad j=r,\ldots,2k,
\end{equation}
holds.

A sequence $\gamma^{(2k+1)}$ is \textbf{positively recursively generated} if, for
$r:=\Rank \gamma^{(2k)}$, one has $H_{r-1}(\gamma^{(2k)})\succ 0$ and the recursion
\eqref{recursion} holds also for $j=2k+1$.

\subsection{Case 1: In \eqref{def:E}, $E=[0,\infty)$}
Let $\widetilde\gamma=\{\widetilde\gamma_t\}_{t\le 3n}$ 
be as in Proposition \ref{prop:one_step_reduction}. The solution to the $E$--TMP for $\widetilde \gamma$ is the following. 

\begin{theorem}[{\cite[Theorems~5.1, 5.3]{CF91}}]
Let $\widetilde\gamma=\{\widetilde\gamma_t\}_{t\le 3n}$.
\begin{enumerate}
\setlength{\itemsep}{10pt}
\item If $n$ is even, then the following are equivalent:
\begin{enumerate}
\setlength{\itemsep}{5pt}
\item $\widetilde\gamma$ admits a $[0,\infty)$--representing measure.
\item $\widetilde\gamma$ admits a $[0,\infty)$--representing $(\Rank \widetilde\gamma)$--atomic measure.
\item $\uH_{3n}(\widetilde\gamma)\succeq 0$, $\uH^{[0,\infty)}_{3n-1}(\widetilde\gamma)\succeq 0$, and
\(
\uh_{3n}^{[0,\infty)}(\widetilde\gamma)\in \cC\big(\uH^{[0,\infty)}_{3n-1}(\widetilde\gamma)\big).
\)
\end{enumerate}

\item If $n$ is odd, then the following are equivalent:
\begin{enumerate}
\setlength{\itemsep}{5pt}
\item $\widetilde\gamma$ admits a $[0,\infty)$--representing measure.
\item $\widetilde\gamma$ admits a $[0,\infty)$--representing $(\Rank \widetilde\gamma^{(3n-1)})$--atomic measure.
\item $\uH_{3n-1}(\widetilde\gamma)\succeq 0$, $\uH^{[0,\infty)}_{3n}(\widetilde\gamma)\succeq 0$, and
\(
\uh_{3n}(\widetilde\gamma)\in \cC\big(\uH_{3n-1}(\widetilde\gamma)\big).
\)
\end{enumerate}
\end{enumerate}

Moreover, in the even case $\Rank \widetilde\gamma\le \frac{3n}{2}+1$, while in the odd case
$\Rank \widetilde\gamma\le \frac{3n+1}{2}$.
Consequently, we obtain a representing measure with at most $(3n+2)$ atoms for $\beta$ from Proposition \ref{prop:one_step_reduction} when $n$ is even
and at most $(3n+1)$ atoms when $n$ is odd.
\end{theorem}

\bigskip

\subsection{Case 2: In \eqref{def:E}, $E=[0,x_2-x_1]\cup[x_3-x_1,\infty)$.}

For brevity, set $c:=x_2-x_1$ and $d:=x_3-x_1$. Let $L_{\widetilde{\gamma}}$ denote the Riesz functional associated with
$\widetilde{\gamma}$.

To solve the TMP for $\widetilde{\gamma}=\{\widetilde{\gamma}_t\}_{t\le 3n}$ on $E$, we combine:
\begin{enumerate}
\item[(i)] the truncated Riesz--Haviland theorem \cite[Theorem~1.2]{CF08}, which reduces the problem to finding an extension
$\widehat{\gamma}$ for which the associated Riesz functional $L_{\widehat{\gamma}}$ is $E$--positive, and
\item[(ii)] the characterization of polynomials nonnegative on $E$ \cite[Theorem~4.1]{KMS05}, namely
\begin{equation}\label{Psatz}
    f|_{E}\ge 0
    \quad\Longleftrightarrow\quad
    f=\sigma_0+\sigma_1 x+\sigma_2(x-c)(x-d)+\sigma_3 x(x-c)(x-d),
\end{equation}
where $\sigma_0,\sigma_1,\sigma_2,\sigma_3\in \sum\mathbb{R}[x]^2$, and each term may be chosen so that its degree is bounded
by $\deg f$.
\end{enumerate}

We also require the following extension criterion for singular Hankel matrices.

\begin{proposition}\label{230810-1952}
Let $\gamma=\{\gamma_t\}_{t\le 2m}$, $m\in\mathbb{N}$, be such that $H_m(\gamma)$ is positive semidefinite and singular.
Then $\gamma$ admits an extension
\[
\widetilde{\gamma}=\{\widetilde{\gamma}_t\}_{t\le 2m+4}
\]
(with new moments $\widetilde{\gamma}_{2m+1},\widetilde{\gamma}_{2m+2},\widetilde{\gamma}_{2m+3},\widetilde{\gamma}_{2m+4}$)
such that $H_{m+2}(\widetilde{\gamma})$ is positive semidefinite if and only if $\gamma$ is positively recursively generated.
In particular, the moments $\widetilde{\gamma}_{2m+1},\widetilde{\gamma}_{2m+2},\widetilde{\gamma}_{2m+3},\widetilde{\gamma}_{2m+4}$  are uniquely
determined.
\end{proposition}

\begin{proof}
The implication $(\Leftarrow)$ follows from \cite[Theorem~2.4]{CF91}, and $(\Rightarrow)$ follows from
\cite[Theorem~2.6]{CF91}.
\end{proof}

We now state explicit criteria for the $E$--TMP for $\widetilde\gamma=\{\widetilde\gamma_t\}_{t\le 3n}$,
distinguishing the parity of $n$.

\begin{theorem}[Even case]\label{230811-0628-even}
Let $\widetilde\gamma=\{\widetilde\gamma_t\}_{t\le 3n}$ and assume that $n$ is even. The following are equivalent:
\begin{enumerate}
\setlength{\itemsep}{5pt}
\item\label{230810-1903-pt1}
$\widetilde\gamma$ admits an $E$--representing measure.
\item\label{230810-1903-pt2}
$\widetilde\gamma$ admits an $E$--representing finitely atomic measure.
\item\label{230810-1903-pt3}
The matrices
\begin{equation}\label{230810-1333}
\uH_{3n}(\widetilde\gamma),\quad
\uH_{3n-1}^{[0,\infty)}(\widetilde\gamma),\quad
\uH_{3n}^{(-\infty,c]\cup[d,\infty)}(\widetilde\gamma),\quad
\uH_{3n}^{[0,c]\cup[d,\infty)}(\widetilde\gamma)
\end{equation}
are positive semidefinite, and one of the following holds:
\begin{enumerate}
\setlength{\itemsep}{5pt}
\item\label{230810-1937-pt1}
All matrices in \eqref{230810-1333} are positive definite.
\item\label{230810-1937-pt2}
Let $H$ be the first matrix in \eqref{230810-1333} that fails to be positive definite. Then
\[
H\in\Bigl\{\uH_{3n}(\widetilde\gamma),\ \uH_{3n}^{(-\infty,c]\cup[d,\infty)}(\widetilde\gamma)\Bigr\}.
\]
The sequence determined by $H$ is positively recursively generated. Moreover, the moments
$\gamma_{3n+1}$ and $\gamma_{3n+2}$ are uniquely determined (and hence so is
$\widehat\gamma=\{\gamma_t\}_{t\le 3n+2}$), and the extended matrices
\begin{equation}\label{230810-1341}
\uH_{3n+2}(\widehat\gamma),\quad
\uH_{3n+1}^{[0,\infty)}(\widehat\gamma),\quad
\uH_{3n+2}^{(-\infty,c]\cup[d,\infty)}(\widehat\gamma),\quad
\uH_{3n+2}^{[0,c]\cup[d,\infty)}(\widehat\gamma)
\end{equation}
are positive semidefinite.
\end{enumerate}
\end{enumerate}
\end{theorem}

\begin{theorem}[Odd case]\label{230811-0628-odd}
Assume the notation above and suppose that $n$ is odd. Then the following are equivalent:
\begin{enumerate}
\setlength{\itemsep}{5pt}
\item\label{230810-1903-pt1-odd}
$\widetilde\gamma$ admits an $E$--representing measure.
\item\label{230810-1903-pt2-odd}
$\widetilde\gamma$ admits an $E$--representing finitely atomic measure.
\item\label{230810-1903-pt3-odd}
The matrices
\begin{equation}\label{230810-1333-odd}
\uH_{3n-1}(\widetilde\gamma),\quad
\uH_{3n}^{[0,\infty)}(\widetilde\gamma),\quad
\uH_{3n-1}^{(-\infty,c]\cup[d,\infty)}(\widetilde\gamma),\quad
\uH_{3n}^{[0,c]\cup[d,\infty)}(\widetilde\gamma)
\end{equation}
are positive semidefinite, and one of the following holds:
\begin{enumerate}
\setlength{\itemsep}{5pt}
\item\label{230810-1937-pt1-odd}
All matrices in \eqref{230810-1333-odd} are positive definite.
\item\label{230810-1937-pt2-odd}
Let $H$ be the first matrix in \eqref{230810-1333-odd} that fails to be positive definite. The sequence determined by $H$ is
positively recursively generated,
and, for the uniquely determined moment $\gamma_{3n+1}$ (and hence
$\widehat\gamma=\{\gamma_t\}_{t\le 3n+1}$), the extended matrices
\begin{equation}\label{230810-1341-odd}
\uH_{3n+1}(\widehat\gamma),\quad
\uH_{3n+1}^{(-\infty,c]\cup[d,\infty)}(\widehat\gamma)
\end{equation}
are positive semidefinite.
\end{enumerate}
\end{enumerate}
\end{theorem}

\begin{proof}[Proof of Theorem~\ref{230811-0628-even}]
The equivalence between \eqref{230810-1903-pt1} and \eqref{230810-1903-pt2} follows from Richter's theorem \cite{Ric}.
The implication $\eqref{230810-1903-pt2}\Rightarrow \eqref{230810-1903-pt3}$ is a consequence of the fact that
singularity of the localizing Hankel matrix $\uH_{3n-1}^{[0,\infty)}(\widetilde\gamma)$ forces
$\widetilde\gamma$ to be positively recursively generated; in particular, $\uH_{3n}(\widetilde\gamma)$ is then
singular as well. An analogous argument applies to the pair of localizing matrices
$\uH_{3n}^{(-\infty,c]\cup[d,\infty)}(\widetilde\gamma)$ and
$\uH_{3n}^{[0,c]\cup[d,\infty)}(\widetilde\gamma)$.

It remains to prove the implication is $\eqref{230810-1903-pt3}\Rightarrow \eqref{230810-1903-pt2}$.
By the truncated Riesz--Haviland theorem \cite[Theorem~1.2]{CF08} together with the Positivstellensatz for $E$
\cite[Theorem~4.1]{KMS05} (cf.\ \eqref{Psatz}), it suffices to extend $\widetilde\gamma$ to $\widehat\gamma=\{\gamma_t\}_{t\le 3n+2}$
by choosing $\gamma_{3n+1}$ and $\gamma_{3n+2}$ so that all matrices in \eqref{230810-1341} are positive semidefinite.

Introduce variables $\mathbf{x}$ and $\mathbf{y}$ representing $\gamma_{3n+1}$ and $\gamma_{3n+2}$, respectively, and set
$\widehat\gamma(\mathbf{x},\mathbf{y})=(\widetilde\gamma,\mathbf{x},\mathbf{y})$. Then
\begin{align*}
\uH_{3n+2}\big(\widehat\gamma(\mathbf{x},\mathbf{y})\big)
&=
\begin{pmatrix}
\uH_{3n}(\widetilde\gamma) & \begin{array}{c} \ast \\ \mathbf{x} \end{array} \\
\begin{array}{cc} \ast & \mathbf{x} \end{array} & \mathbf{y}
\end{pmatrix},\\[0.6em]
\uH^{[0,\infty)}_{3n+1}\big(\widehat\gamma(\mathbf{x},\mathbf{y})\big)
&=
\begin{pmatrix}
\uH^{[0,\infty)}_{3n-1}(\widetilde\gamma) & \uh^{[0,\infty)}_{3n}(\widetilde\gamma)\\[0.4em]
\big(\uh^{[0,\infty)}_{3n}(\widetilde\gamma)\big)^{T} & \mathbf{x}
\end{pmatrix},\\[0.6em]
\uH_{3n+2}^{(-\infty,c]\cup[d,\infty)}\big(\widehat\gamma(\mathbf{x},\mathbf{y})\big)
&=
\begin{pmatrix}
\uH_{3n}^{(-\infty,c]\cup[d,\infty)}(\widetilde\gamma) &
\begin{array}{c}
\ast\\
\mathbf{x}-(c+d)\gamma_{3n}+cd\,\gamma_{3n-1}
\end{array}\\[0.7em]
\begin{array}{cc}
\ast &
\mathbf{x}-(c+d)\gamma_{3n}+cd\,\gamma_{3n-1}
\end{array}
&
\mathbf{y}-(c+d)\mathbf{x}+cd\,\gamma_{3n}
\end{pmatrix},\\[0.6em]
\uH_{3n+2}^{[0,c]\cup[d,\infty)}\big(\widehat\gamma(\mathbf{x},\mathbf{y})\big)
&=
\begin{pmatrix}
\uH_{3n}^{[0,c]\cup[d,\infty)}(\widetilde\gamma) & \uh_{3n+1}^{[0,c]\cup[d,\infty)}(\widetilde\gamma)\\[0.4em]
\big(\uh_{3n+1}^{[0,c]\cup[d,\infty)}(\widetilde\gamma)\big)^{T} &
\mathbf{x}-(c+d)\gamma_{3n}+cd\,\gamma_{3n-1}
\end{pmatrix},
\end{align*}
where each entry marked by $\ast$ is determined uniquely by $\widetilde\gamma$.

\smallskip
Assume first that we are in case \eqref{230810-1937-pt1}. Choose $x$ sufficiently large so that both matrices
$\uH^{[0,\infty)}_{3n+1}\big(\widehat\gamma(x,\mathbf{y})\big)$ and
$\uH_{3n+2}^{[0,c]\cup[d,\infty)}\big(\widehat\gamma(x,\mathbf{y})\big)$ are positive definite. Then choose $y$ sufficiently large so that the matrices
$\uH_{3n+2}\big(\widehat\gamma(x,y)\big)$ and
$\uH_{3n+2}^{(-\infty,c]\cup[d,\infty)}\big(\widehat\gamma(x,y)\big)$ are positive definite as well.


\smallskip
Assume next that we are in case \eqref{230810-1937-pt2}. By Proposition~\ref{230810-1952}, the first singular matrix $H$
admits a unique chain of positive semidefinite Hankel extensions, which produces unique candidates for $\gamma_{3n+1}$ and
$\gamma_{3n+2}$. One then checks whether the remaining matrices in \eqref{230810-1341} are positive semidefinite.

We claim that if $H=\uH_{3n-1}^{[0,\infty)}(\widetilde\gamma)$ or $H=\uH_{3n}^{[0,c]\cup[d,\infty)}(\widetilde\gamma)$,
then no representing measure exists. Indeed:
\begin{enumerate}
\setlength{\itemsep}{5pt}
\item If $H=\uH_{3n-1}^{[0,\infty)}(\widetilde\gamma)$, then by definition $\uH_{3n}(\widetilde\gamma)$ is positive definite.
However, the first $\frac{3n}{2}$ columns of any extension $\uH_{3n+1}^{[0,\infty)}(\widehat\gamma)$ coincide with columns of
$\uH_{3n}(\widetilde\gamma)$, which forces $\uH_{3n}(\widetilde\gamma)$ to have a nontrivial kernel-a contradiction.
\item If $H=\uH_{3n}^{[0,c]\cup[d,\infty)}(\widetilde\gamma)$, then by definition the other three matrices in
\eqref{230810-1341} are positive definite. But the first $\frac{3n}{2}-1$ columns of any extension
$\uH_{3n+2}^{(-\infty,c]\cup[d,\infty)}(\widehat\gamma)$ coincide with columns of
$\uH_{3n}^{[0,c]\cup[d,\infty)}(\widetilde\gamma)$, again forcing a nontrivial kernel-a contradiction.
\end{enumerate}
This completes the proof.
\end{proof}

\begin{proof}[Proof of Theorem~\ref{230811-0628-odd}]
The equivalence between \eqref{230810-1903-pt1-odd} and \eqref{230810-1903-pt2-odd} follows from Richter's theorem.
The implication $\eqref{230810-1903-pt2}\Rightarrow \eqref{230810-1903-pt3}$ is clear.
It remains to prove $\eqref{230810-1903-pt3-odd}\Rightarrow \eqref{230810-1903-pt2-odd}$.
By \cite[Theorem~1.2]{CF08} and \cite[Theorem~4.1]{KMS05}, it suffices to extend $\widetilde\gamma$ to
$\widehat\gamma=\{\gamma_t\}_{t\le 3n+1}$ by choosing $\gamma_{3n+1}$ so that the matrices in \eqref{230810-1333-odd} and
\eqref{230810-1341-odd} are positive semidefinite.

Let $\mathbf{x}$ represent $\gamma_{3n+1}$ and write $\widehat\gamma(\mathbf{x})=(\widetilde\gamma,\mathbf{x})$. Then
\begin{align*}
\uH_{3n+1}\big(\widehat\gamma(\mathbf{x})\big)
&=
\begin{pmatrix}
\uH_{3n-1}(\widetilde\gamma) & \uh_{3n}(\widetilde\gamma)\\
\uh_{3n}(\widetilde\gamma)^{T} & \mathbf{x}
\end{pmatrix},\\[0.6em]
\uH_{3n+1}^{(-\infty,c]\cup[d,\infty)}\big(\widehat\gamma(\mathbf{x})\big)
&=
\begin{pmatrix}
\uH_{3n-1}^{(-\infty,c]\cup[d,\infty)}(\widetilde\gamma) &
\uh_{3n}^{(-\infty,c]\cup[d,\infty)}(\widetilde\gamma)\\
\big(\uh_{3n}^{(-\infty,c]\cup[d,\infty)}(\widetilde\gamma)\big)^T &
\mathbf{x}-(c+d)\gamma_{3n}+cd\,\gamma_{3n-1}
\end{pmatrix}.
\end{align*}

Assume first that we are in case \eqref{230810-1937-pt1-odd}. Choose $x$ sufficiently large so that both matrices above
 are positive definite. 

Assume next that we are in case \eqref{230810-1937-pt2-odd}. Proposition~\ref{230810-1952} implies that $H$ admits a unique chain
of positive semidefinite Hankel extensions, which yields the unique candidate for $\gamma_{3n+1}$. One then checks whether the
 matrices in \eqref{230810-1341-odd} are positive semidefinite. 
\end{proof}


\begin{thebibliography}{Arv666}
\bibitem[AK62]{AK62}
N.\ I.\ Akhiezer, M.\ Krein,
\textit{Some Questions in the Theory of Moments},
Translations of Mathematical Monographs, Vol.\ 2,
American Mathematical Society, Providence, 1962.

\bibitem[BBS+]{BBS+}
L.\ Baldi, B.\ Blekherman, R.\ Sinn,
\textit{Nonnegative Polynomials and Truncated Moment Problem on Curves},
arXiv preprint, \url{https://arxiv.org/abs/2407.06017}

\bibitem[BZ21]
{BZ21}
	A.\ Bhardwaj, {{A.\ Zalar}}:
		The tracial moment problem on quadratic varieties, 
			\textit{J.\ Math.\ Anal.\ Appl.}  498 (2021) 39 pp.\

\bibitem[BF20]{BF20}
G.\ Blekherman, L.\ Fialkow,
\textit{The core variety and representing measures in the truncated moment problem},
J.\ Operator Theory 84 (2020), 185--209.

\bibitem[CF91]{CF91}
R.\ Curto, L.\ Fialkow,
\textit{Recursiveness, positivity, and truncated moment problems},
Houston J.\ Math.\ 17 (1991), 603--635.

\bibitem[CF96]{CF96}
R.\ Curto, L.\ Fialkow,
\textit{Solution of the truncated complex moment problem for flat data},
Mem.\ Amer.\ Math.\ Soc.\ 119 (1996).

\bibitem[CF02]{CF02}
R.\ Curto, L.\ Fialkow,
\textit{Solution of the singular quartic moment problem},
J.\ Operator Theory 48 (2002), 315--354.

\bibitem[CF04]{CF04}
R.\ Curto, L.\ Fialkow,
\textit{Solution of the truncated parabolic moment problem},
Integral Equ.\ Operator Theory 50 (2004), 169--196.

\bibitem[CF05]{CF05}
R.\ Curto, L.\ Fialkow,
\textit{Solution of the truncated hyperbolic moment problem},
Integral Equ.\ Operator Theory 52 (2005), 181--218.

\bibitem[CF08]{CF08}
R.\ Curto, L.\ Fialkow,
\textit{An analogue of the Riesz--Haviland theorem for the truncated moment problem},
J.\ Funct.\ Anal.\ 255 (2008), 2709--2731.

\bibitem[CF13]{CF13}
R.\ Curto, L.\ Fialkow,
\textit{Recursively determined representing measures for bivariate truncated moment sequences},
J.\ Operator Theory 70(2) (2013), 401--436.

\bibitem[CY16]{CY16}
R.\ Curto, S.\ Yoo,
\textit{Concrete solution to the nonsingular quartic binary moment problem},
Proc.\ Amer.\ Math.\ Soc.\ 144 (2016), 249--258.



\bibitem[EKT25]
{EKT25}
C.\ Emary, D.P.\ Kimsey, C.P.\ Tantalakis, 
 \textit{On a solution of the multidimensional truncated moment problem on vertices of the hypercube based on Bell inequalities},
 {Trans.\ Amer.\ Math.\ Soc.} 378 (2025), 6831--6855.

\bibitem[F95]{F95}
L.\ Fialkow,
\textit{Positivity, extensions and the truncated complex moment problem},
Contemporary Math.\ 185 (1995), 133--150.

\bibitem[Fia11]{Fia11}
L.\ Fialkow,
\textit{Solution of the truncated moment problem with variety $y=x^3$},
Trans.\ Amer.\ Math.\ Soc.\ 363 (2011), 3133--3165.

\bibitem[Fia15]{Fia15}
L.\ Fialkow,
\textit{The truncated moment problem on parallel lines},
in: \textit{The Varied Landscape of Operator Theory},
Theta Found.\ Internat.\ Book Ser.\ Math.\ Texts, Vol.\ 20,
Amer.\ Math.\ Soc., 2015, 99--116.

\bibitem[FN]{fn}
L.\ Fialkow, J.\ Nie,
\textit{Positivity of Riesz functionals and solutions of quadratic and quartic moment problems},
J.\ Funct.\ Anal.\ 258 (2010), 328--356.

\bibitem[FZ+]{fz}
L.\ Fialkow, A.\ Zalar,
\textit{A core variety approach to the pure $Y=X^d$ truncated moment problem: part 1},
arXiv preprint, \url{https://arxiv.org/abs/2508.10375}.

\bibitem[Ioh82]{Ioh82}
I.\ S.\ Iohvidov,
\textit{Hankel and Toeplitz Matrices and Forms: Algebraic Theory},
Birkh\"auser, Boston, 1982.

\bibitem[KN77]{KN77}
K.\ G.\ Krein, A.\ A.\ Nudelman,
\textit{The Markov Moment Problem and Extremal Problems},
Translations of Mathematical Monographs,
American Mathematical Society, 1977.

\bibitem[KMS05]{KMS05}
S.\ Kuhlmann, M.\ Marshall, N.\ Schwartz,
\textit{Positivity, sums of squares and the multidimensional moment problem II},
Adv.\ Geom.\ 5 (2005), 583--607.

\bibitem[KZ+]{KZ}
M.\ Kummer, A.\ Zalar,
\textit{Positive polynomials and the truncated moment problem on plane cubics},
arXiv preprint, \url{https://arxiv.org/abs/2508.13850v1}.

\bibitem[Ric]{Ric}
H.\ Richter,
\textit{Parameterfreie Absch\"at}zung und Realisierung von Erwartungswerten,
Bl.\ der Deutsch.\ Ges.\ Versicherungsmath.\ 3 (1957), 147--161.

\bibitem[Sch17]{sch}
K.\ Schm\"udgen,
\textit{The Moment Problem},
Graduate Texts in Mathematics, Vol.\ 277,
Springer, 2017.

\bibitem[ST]{ST}
J.\ Shohat, J.\ Tamarkin,
\textit{The Problem of Moments},
Math.\ Surveys I, Amer.\ Math.\ Soc., Providence, 1943.

\bibitem[Wol]{Wol}
Wolfram Research, Inc.,
\textit{Mathematica}, Version 14.2,
Wolfram Research, Inc., Champaign, IL, 2025.

\bibitem[YZ24]{yz}
S.\ Yoo, A.\ Zalar,
\textit{The truncated moment problem on reducible cubic curves I: Parabolic and Circular type relations},
Complex Anal.\ Oper.\ Theory 18 (2024), 111, 54 pp.

\bibitem[YZ+]{yz+}
S.\ Yoo, A.\ Zalar,
\textit{Bivariate Truncated Moment Sequences with the Column Relation $XY=X^m+q(X)$, with $q$ of degree $m-1$},
arXiv preprint, \url{https://arxiv.org/pdf/2412.21020}.

\bibitem[YZ++]
{YZ25+}
{S. Yoo}, {{ A.\ Zalar}}: 
\textit{Constructive approach to the truncated moment problem on reducible cubic curves: Hyperbolic type relations}, 
{arXiv preprint}
\url{https://arxiv.org/pdf/2510.15131}

\bibitem[Zal21]{Zal21}
A.\ Zalar,
\textit{The truncated Hamburger moment problem with gaps in the index set},
Integral Equ.\ Operator Theory 93 (2021), 36 pp.

\bibitem[Zal22a]{Zal22a}
A.\ Zalar,
\textit{The truncated moment problem on the union of parallel lines},
Linear Algebra Appl.\ 649 (2022), 186--239.

\bibitem[Zal22b]{Zal22b}
A.\ Zalar,
\textit{The strong truncated Hamburger moment problem with and without gaps},
J.\ Math.\ Anal.\ Appl.\ 516 (2022), 21 pp.

\bibitem[Zal23]{Zal23}
A.\ Zalar,
\textit{The truncated moment problem on curves $y=q(x)$ and $yx^\ell=1$},
Linear and Multilinear Algebra (2023), 45 pp.


\end{thebibliography}
\end{document}